\documentclass[12pt]{amsart}
\usepackage[margin=0.78in]{geometry}

\usepackage[english]{babel}
\usepackage{amsmath,amssymb,amsfonts,amsthm,enumerate}
\usepackage{url}
\usepackage{graphicx,epstopdf,color}

\numberwithin{equation}{section}

\newtheorem{theorem}{Theorem}[section]
\newtheorem{lemma}[theorem]{Lemma}
\newtheorem{definition}[theorem]{Definition}
\newtheorem{remark}[theorem]{Remark}
\newtheorem{proposition}[theorem]{Proposition}
\newtheorem{corollary}[theorem]{Corollary}

\newcommand{\R}{\mathbb{R}}
\newcommand{\Z}{\mathbb{Z}}
\newcommand{\N}{\mathbb{N}}

\renewcommand{\epsilon}{\varepsilon}

\renewcommand{\leq}{\leqslant}
\renewcommand{\le}{\leqslant}

\renewcommand{\ge}{\geqslant}

\title{Chaotic orbits for systems of nonlocal equations}

\author{Serena Dipierro}
\address[Serena Dipierro]{School of Mathematics and Statistics,
University of Melbourne, 813 Swanston St, Parkville VIC 3010, Australia,
and School of Mathematics and Statistics,
University of Western Australia,
35 Stirling Highway,
Crawley, Perth
WA 6009, Australia,
and Weierstra{\ss} Institut f\"ur Angewandte Analysis
und Stochastik, Mohrenstrasse 39, 10117 Berlin, Germany}
\email{serena.dipierro@ed.ac.uk}

\author{Stefania Patrizi}
\address[Stefania Patrizi]{The University of Texas at Austin,
Department of Mathematics, 2515 Speedway, Austin, TX 78751, USA}
\email{spatrizi@math.utexas.edu}

\author{Enrico Valdinoci}
\address[Enrico Valdinoci]{School of Mathematics and Statistics,
University of Melbourne, 813 Swanston St, Parkville VIC 3010, Australia,
and School of Mathematics and Statistics,
University of Western Australia,
35 Stirling Highway,
Crawley, Perth
WA 6009, Australia,
and Weierstra{\ss} Institut f\"ur Angewandte Analysis
und Stochastik, Mohrenstrasse 39, 10117 Berlin, Germany,
and 
Universit\`a degli studi di Milano,
Dipartimento di Matematica, 
Via Saldini 50, 20133 Milan, Italy, and
Istituto di Matematica Applicata e Tecnologie Informatiche,
Consiglio Nazionale delle Ricerche,
Via Ferrata 1, 27100 Pavia, Italy}
\email{enrico@mat.uniroma3.it}

\begin{document}

\begin{abstract}
We consider a system of nonlocal equations driven by a perturbed
periodic potential. 
We construct multibump solutions that connect
one integer point to another one in a prescribed way.
In particular, heteroclinic, homoclinic 
and chaotic trajectories are constructed.

This is the first attempt to consider a nonlocal version
of this type of
dynamical systems in a variational setting and the first result
regarding symbolic dynamics in a fractional framework.
\end{abstract}

\maketitle

\section{Introduction}

Goal of this paper is to construct heteroclinic and multibumps orbits
for a class of systems of integrodifferential equations.
The forcing term of the equation comes from a multiwell potential
(for simplicity, say periodic and centered at integer points,
though more general potential with a discrete
set of minima may be similarly taken into account).

The solutions constructed connect the equilibria of the potential
in a rather arbitrary way and thus reveal a chaotic behavior 
of the problem into consideration.\medskip

More precisely, the mathematical framework that we consider is the following.
Given~$s\in\left(\frac12,1\right)$, we consider
an interaction kernel~$K:\R\to[0,+\infty]$,
satisfying the structural assumptions~$K(-x)=K(x)$,
\begin{equation}\label{KERNEL}
\frac{\theta_0 \;(1-s)\;\chi_{[-\rho_0,\rho_0]} (x)}{|x|^{1+2s}}
\le K(x)\le \frac{\Theta_0 \;(1-s)}{|x|^{1+2s}}
\end{equation}
for some~$\rho_0\in (0,1]$ and~$\Theta_0\ge\theta_0>0$, and
\begin{equation}\label{KERNEL:2}
|\nabla K(x)|\le \frac{\Theta_1}{|x|^{2+2s}}
\end{equation}
for some~$\Theta_1>0$.

We consider\footnote{Of course, \label{FOO:678uhauu}
for a fixed~$s\in\left(\frac12,1\right)$,
the quantity~$(1-s)$ in~\eqref{KERNEL} does not play any role,
since it can be reabsorbed into~$\theta_0$ and~$\Theta_0$. The advantage
of extrapolating this quantity explicitly is that, in this way,
all the quantities involved in this paper will be bounded uniformly as~$s\to1$,
i.e., fixed~$s_0\in \left(\frac12,1\right)$ and given any~$s\in[s_0,1)$,
the constants will depend only on~$s_0$, and not explicitly on~$s$.
This technical improvement plays often an important role
in the study of nonlocal equations, see e.g.~\cite{byapp}, and allows
us to comprise
the classical case of the second derivative as a limit case of our results.}
the energy associated to such interaction kernel: namely,
for any measurable function~$Q:\R\to\R^n$, with~$n\in\N$, $n\ge1$,
we define
\begin{equation}\label{def E:Q}
E(Q) := \iint_{\R\times\R} K(x-y)\,\big|Q(x)-Q(y) \big|^2\,dx\,dy.\end{equation}
Our goal is to take into account the 
integrodifferential equation
satisfied by the critical points of~$E$. 

For this, given
an interval~$J\subseteq\R$,
a measurable function~$Q:\R\to\R^n$, with~$E(Q)<+\infty$, 
and~$f\in L^1(J,\R^n)$ we say that~$Q$
is a solution of
\begin{equation} \label{EQUAZIONE}
{\mathcal{L}}(Q)(x)+f(x)=0\end{equation}
if
\begin{equation}\label{U756GA}
2\iint_{\R\times\R} K(x-y)\,
\big(Q(x)-Q(y) \big)\cdot \big(\psi(x)-\psi(y) \big)
\,dx\,dy +\int_\R f(x)\cdot\psi(x)\,dx=0,\end{equation}
for any~$\psi\in C^\infty_0(J,\R^n)$. We remark that~\eqref{EQUAZIONE}
provides a single equation for~$n=1$
and a system\footnote{As a matter of fact, we observe that, with
minor modifications of our methods, one can also consider the
case in which each equation of the system is driven
by an integrodifferential operator of different order.}
of equations for $n\ge2$.\medskip

In the strong version, the operator~${\mathcal{L}}(Q)$
may be interpreted as the integrodifferential operator
$$ 4\int_\R K(x-y)\,(Q(x)-Q(y))\,dy,$$
with the singular integral taken in its principal value sense.\medskip

The prototype of the interaction kernel that we have in mind
is~$K(x):=\frac{1-s}{|x|^{1+2s}}$. In this case,
the operator~${\mathcal{L}}(Q)$ in~\eqref{EQUAZIONE}
is (up to multiplicative constants) the fractional Laplacian~$(-\Delta)^s Q$.

The setting considered in~\eqref{KERNEL}
is very general, since it comprises 
operators which are not necessarily homogeneous or isotropic.
\medskip

The particular equation that we consider in this paper is
\begin{equation}\label{EQUAZ}
{\mathcal{L}}(Q)(x) + a(x)\,\nabla W(Q(x)) =0 \qquad{\mbox{
for any }} x\in\R.\end{equation}
We suppose that~$W
\in C^{1,1}(\R^n)$
and that it is periodic of period~$1$,
that is~$W(\tau+\zeta)=W(\tau)$ for any~$\tau\in\R^n$ and~$\zeta\in\Z^n$. 

We also assume that
the minima of~$W$ are attained at the integers: namely we suppose that
\begin{equation}\label{ZERI di W}
{\mbox{$W(\zeta)=0$ 
for any~$\zeta\in\Z^n$ and that~$W(\tau)>0$
for any~$\tau\in\R^n\setminus\Z^n$.}}\end{equation}
Also, we suppose that the minima of~$W$ are ``nondegenerate''. More
precisely, we assume that
there exist
$r\in (0,\,1/4]$, $c_0 \in (0,1)$ and~$C_0\in(1,+\infty)$
such that
\begin{equation}\label{GROW}
c_0\, |\tau|^2 \le W(\tau)\le C_0\,|\tau|^2 \qquad{\mbox{
for any }}\tau\in B_r.
\end{equation}
These assumptions on~$W$ are indeed rather general and fit into the
well-established theory of multiwell potentials.

The function~$a$ can be considered
as a perturbation of the potential,
and many structural results
hold under the basic
conditions that~$a\in C^1(\R)$ with~$a'\in L^\infty(\R)$, and
that there exist~$\underline{a}\in (0,1)$ and~$\overline{a}\in(1,+\infty)$
such that
\begin{equation}\label{GROW:2}
\underline{a}\le a(x)\le \overline{a} \qquad{\mbox{
for any }}x\in\R.
\end{equation}
On the other hand,
to construct unstable orbits, one also assumes that $a$ satisfies
a ``nondegeneracy condition''. Several general hypotheses on~$a$ could
be assumed for this scope (see e.g. page~227 in~\cite{rabi2}),
but, to make a simple and concrete example, we stick to the case in which
\begin{equation}\label{ASS:A:NONDEG}
a(x):= a_1+a_2\cos(\epsilon x),
\end{equation}
with~$\epsilon>0$ to be taken suitably small
and~$a_1>a_2>0$ (to be consistent with~\eqref{GROW:2} one can take~$a_1:=
(\overline{a}+\underline{a})/2$ and~$a_2:=(\overline{a}-\underline{a})/2$).

Notice that when~$\epsilon=0$, the perturbation function~$a$ reduces to
a constant and thus it has no effect on the structure of
the solutions of~\eqref{EQUAZ}. On the other hand, we will show 
that for small~$\epsilon$ the perturbation~$a$ produces
a variety of geometrically very different solutions. Namely,
under the conditions above, we construct solutions of~\eqref{EQUAZ}
which connect chains of integers, thus proving a sort of ``chaotic''
behavior for this type of solutions (roughly speaking,
the sequences of integers can be arbitrarily prescribed in a given class,
thus providing a ``symbolic dynamics'').
The behavior of this chaotic trajectories is depicted in Figure~\ref{F:CH}.

\begin{figure}
    \centering
    \includegraphics[width=16.8cm]{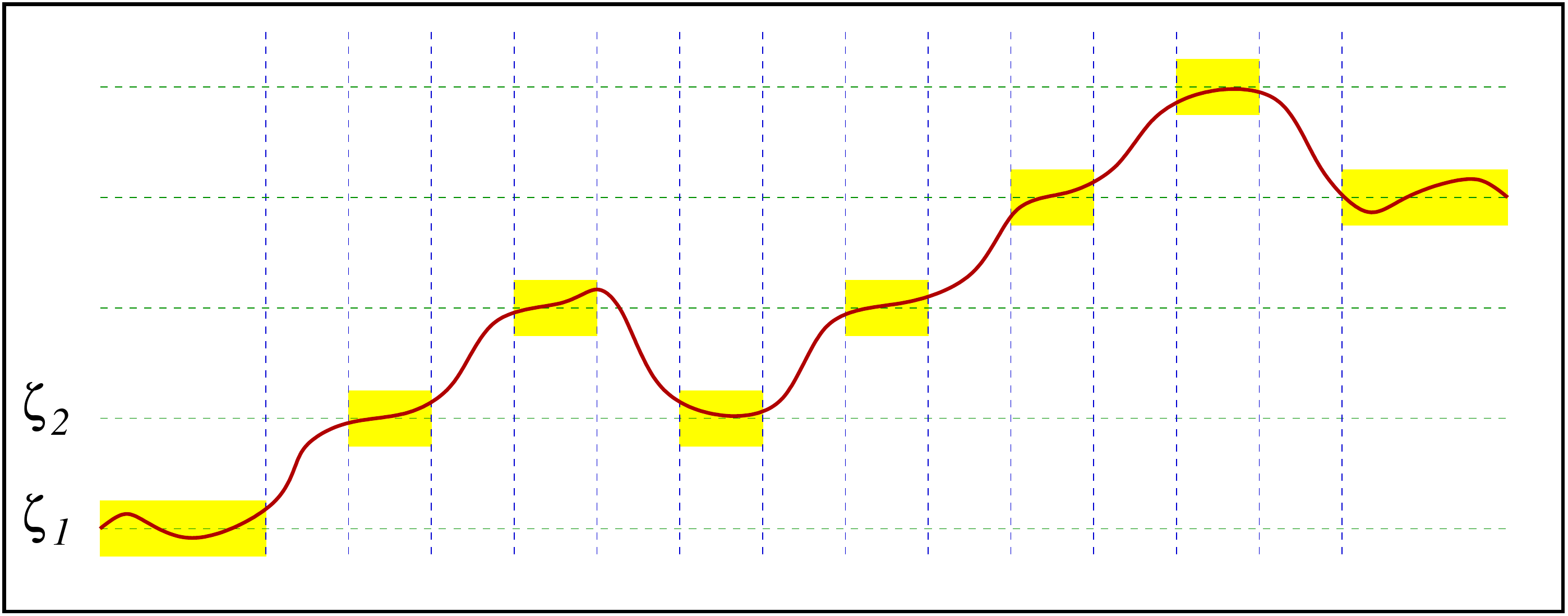}
    \caption{A chaotic trajectory.}
    \label{F:CH}
\end{figure}

More precisely, the main result that we prove in this paper
is the following:

\begin{theorem}\label{TH:MAIN}
Let~$\zeta_1\in\Z^n$ and~$N\in\N$.
There exist~$\zeta_2,\dots,\zeta_N\in\Z^n$
and~$b_1,\dots,b_{2N-2}\in\R$, with~$b_{i+1} \ge b_i + 3$
for all~$i=1,\dots,2N-3$, and a solution~$Q_*$ of~\eqref{EQUAZ} such that
\begin{eqnarray*}
&& \zeta_{i+1}\ne\zeta_i {\mbox{ for any }} i\in\{1,\dots,N-1\},\\
&& \lim_{x\to-\infty} Q_*(x)=\zeta_1,\\
&& \sup_{x\in(-\infty,b_1]} |Q_*(x)-\zeta_1|\le\frac14,\\
&& \sup_{x\in[b_{2i},b_{2i+1}]} |Q_*(x)-\zeta_{i+1}|\le\frac14\qquad{\mbox{
for all }}i=1,\dots, N-2,\\
&& \sup_{x\in[b_{2N-2},+\infty)} |Q_*(x)-\zeta_N|\le\frac14\\
{\mbox{and }}&& \lim_{x\to+\infty} Q_*(x)=\zeta_N.
\end{eqnarray*}
\end{theorem}

More quantitative versions of Theorem~\ref{TH:MAIN}
will be given in the forthcoming
Theorems~\ref{HET}
and~\ref{CHS}.\medskip

The result contained in Theorem~\ref{TH:MAIN} may be seen
as the first attempt in the literature to deal with
heteroclinic, homoclinic and chaotic orbits for systems of
equations driven by fractional operators (as a matter of fact,
to the best of our knowledge, Theorem~\ref{TH:MAIN} is new
even in the case of a single equation with the fractional 
Laplacian).

For local equations, the study of these types of orbits
has a long and celebrated tradition and the local counterpart of
Theorem~\ref{TH:MAIN} is a celebrated result in~\cite{rabi}
(see also~\cite{JM, MR1143210, MR1304144, MR1304144, bessi, MR1422191, MR1433175, MR1473371, berti, alessio, rabi2, MR1804957} and the references
therein for important related results).

We point out that the nonlocal character of the equation
generates several difficulties in the construction 
of the connecting orbits, since all the variational methods
available in the literature are deeply based on the possibility
of ``glueing'' trajectories to provide admissible competitors.
Of course, in the nonlocal case this glueing procedure
is more problematic, since the energy is affected by the
nonlocal interactions.\medskip

In the nonlocal case, as far as we know, multibump solutions
have not been studied in the existing literature. In the homogeneous
case (i.e. when~$a$ is constant), heteroclinic solutions have been constructed in~\cite{MR3081641,
MR3280032, 2015arXiv151002812C}, but the methods used there do not easily extend
to inhomogeneous cases (since sliding methods and extension
techniques are taken into account) and cannot lead to
the construction of chaotic trajectories.
In particular, the reader can compare Theorem~\ref{TH:MAIN} here with
Theorem~1 in~\cite{MR3081641},
Theorem~2.4 in~\cite{MR3280032}
or Theorem~1 in~\cite{2015arXiv151002812C}: all these results
provide the existence of transition layers in one dimension for
spatially homogeneous doublewell potentials (and in this sense are
related to Theorem~\ref{TH:MAIN} here when~$N=2$), but the methods heavily use
maximum principle or extension techniques, so they cannot be easily adapted to
consider higher dimensional cases and inhomogeneous cases 
(also, extensions methods cannot be applied for general interaction kernels).
\medskip

Also, in the framework of the existing literature,
this paper is the first attempt to
combine the very prolific variational techniques
used in dynamical systems to construct special types of orbits
with the abundant new tools arising in the study of
nonlocal integrodifferential equations.\medskip

In this sense, we are also confident that the results of this
paper can be stimulating for both the scientific communities in
dynamical systems and in partial differential equations
and they can trigger new research in this field in the near future.\medskip

{F}rom the point of view of the applications,
for us, one of the main motivations for studying nonlocal
variational problems as in~\eqref{EQUAZ} came from similar
equations arising in the study of atom dislocations
in crystals and in nonlocal phase transition models,
see e.g.~\cite{MR2231783, MR2946964, MR2851899, MR3259559, MR3296170, MR3338445, MR3334171}
and~\cite{MR2948285, MR3081641, MR3280032, 2015arXiv151002812C}.\medskip

Important connections between nonlocal diffusion and dynamical
systems occur also in several other areas of contemporary research,
such as in plasma physics, see e.g.~\cite{castillo}.
\medskip

The rest of the paper is organized as follows. First, in Section~\ref{JAK:HE} -- which
can of course be easily skipped by the expert reader --
we give some heuristic comments on the proof
of Theorem~\ref{TH:MAIN}, trying to elucidate the role played
by the modulation function~$a$ introduced in~\eqref{ASS:A:NONDEG}.
 
In Section~\ref{SE1}
we collect some simple technical lemmata
and in Section~\ref{SeR} we introduce the basic regularity estimates
needed for our purposes. Then,
in Section~\ref{Se2}, we develop the theory
of the nonlocal glueing arguments. In a sense, this part
contains the many novelties
with respect to the classical case, since the classical variational
methods fully exploit
several glueing arguments that are very sensitive to the local
behavior of the energy functional.

The use of the glueing results is effectively implemented
in Section~\ref{CLEAN:SECT}, which contains the new notion
of clean intervals and clean points in this framework.
Roughly speaking, in the classical case,
having two trajectories that meet allows simple glueing
methods to work in order to construct competitors. In our
case, to perform the glueing methods,
we need to attach the trajectories
in an ``almost tangent'' way, and keeping the trajectories
close in Lipschitz norm for a sufficiently large interval.
This phenomenon clearly reflects the nonlocal character of
the problem and requires the definitions and methods introduced in this section.

In Section~\ref{MI99} we develop the minimization theory
for the nonlocal energy under consideration. Differently from the
classical case, this part has to join a suitable regularity theory,
in order to obtain uniform estimates on the nonlocal terms of the energy.

The stickiness properties of the energy minimizers
(i.e., the fact that minimizing orbits stay close to the integer points
once they get sufficiently close to them)
is then discussed in Section~\ref{MI100}. This property is
based on the comparison of the energy with suitable competitors
and thus it requires the nonlocal glueing arguments introduced in
Section~\ref{Se2} and the notion of clean intervals given in Section~\ref{CLEAN:SECT}.

Section~\ref{JAH:SS1} deals with the construction of
heteroclinic orbits: namely, for any integer point,
we define the set of admissible integers that can be connected
with the first one by a heteroclinic orbit (indeed, we will show
that this admissible family contains at least two elements).

In Section~\ref{JAH:SS2}, we complete the proof of Theorem~\ref{TH:MAIN}
by constructing the desired
chaotic orbits.

\section{A few comments on the proof of Theorem~\ref{TH:MAIN} and on the role
of the modulation function~$a$}\label{JAK:HE}

The proof of Theorem~\ref{TH:MAIN} is variational and it can be better
understood by thinking first to the case~$N=2$, i.e. when only one transition
from one integer to another takes place.
In this case, one first considers
a constrained minimization problem, namely one minimizes the action
functional among all the trajectories which are forced to stay sufficiently close to
the first integer in~$(-\infty,b_1]$
and sufficiently close to
the second integer in~$[b_2,+\infty)$
(the formal details
of this constrained minimization argument will be given in Section~\ref{MI99}).
The goal is, in the end, to choose~$b_2>b_1$ in a suitable way
for which the constrained minimal trajectory does not touch the barrier,
hence it is a ``free'' minimizer and so a solution of the desired equation.

To this aim, the appropriate choice of~$b_1$ and~$b_2$ has to take
advantage of the small, but not negligible, oscillations 
of the potential induced by the modulating function~$a$ in~\eqref{ASS:A:NONDEG}.
Roughly speaking, the points~$b_1$ and~$b_2$ will be chosen sufficiently close
to the points in which~$a$ takes its maximal value, say at distance close
to (a multiple of) the period of~$a$, or more generally to the distance between
two wells of~$a$. 

In this way,
for a minimal trajectory it is not convenient to put its ``transition
from one integer to the other'' too close to the constraints. Indeed,
such transition pays energy in virtue of the potential. So, if the transition
occurs too close to~$b_1$, one can consider the translation of the orbit to the right.
Such translated orbit will place the transition in the ``lowest well'' of the 
modulating function~$a$ and so it will pay less potential energy
(the energy coming from particle interaction is on the other hand
invariant under translation). In this way, we see that
the translated orbit would have less energy than the original one,
thus providing a contradiction with the minimality assumption (to facilitate the intuition,
one can look at Figure~\ref{TODO}).

\begin{figure}
    \centering
    \includegraphics[width=16.8cm]{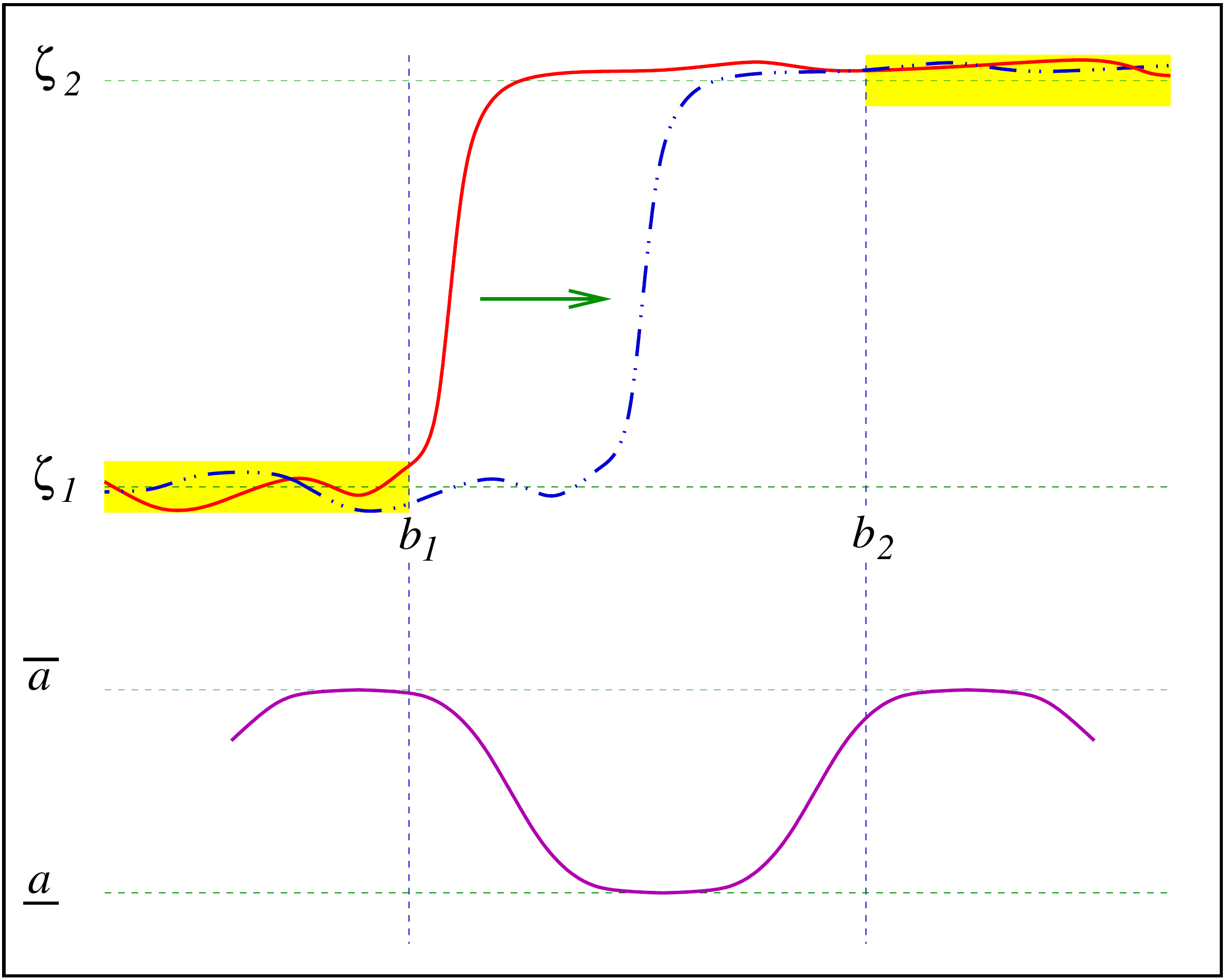}
    \caption{The role of the modulating function~$a$ (below),
compared with a trajectory with transition too close to~$b_1$ and its
translation (above).}
    \label{TODO}
\end{figure}

We stress that
the nondegeneracy of the function~$a$ (that is the fact that~$a$
possesses suitable ``hills and valleys'' in its graph)
is indeed crucial in order to
perform this variational construction, since it is exactly the ingredient
used to allow this energy decreasing under appropriate translations.\medskip

Once the transition is set sufficiently far from the constraint, one
has to perform suitable cut-and-paste arguments to check that
the remaining parts of the trajectory approach the equilibria sufficiently fast,
namely the distance from the two limit integers becomes very fast much smaller than
the prescription given by the initial constraint and so the trajectory is
a true, unconstrained, minimizer.

The choice of~$b_1$ and~$b_2$ in terms of the function~$a$ will
be analytically described in~\eqref{89uIJBAAJKAlAK} and the free minimization
procedure is discussed in details in Section~\ref{JAH:SS1}.\medskip

The case in which~$N\ge3$, i.e. when multibumps arise, is technically
more delicate, since different situations must be taken into account
(according to where the touching with the constraint may occur). Also, when~$N\ge3$,
global translations are not allowed, since they are not compatible with
oscillating constraints, and therefore cut-and-paste arguments must be
performed together with a local translation procedure. Nevertheless,
in spite of these additional difficulties, one may still think that
the role of the modulation given by~$a$ is to make the transitions
near the multiple constraints ``too expensive''. For this, once again,
one has to place the constraint points~$b_1,\dots,b_{2N-2}$ sufficiently
close to the maxima of~$a$, so that the transition will have the tendency
to occur away from them. The analytic choice of these points will be made in~\eqref{DEF:b:sp}.\medskip

We remark that both the local and the nonlocal case share the variational
idea of looking for constrained minimal orbits and then proving that they are
in fact unconstrained minimizers -- of course, in the nonlocal case
the action functional is different than in the local case and it takes
into account an interaction energy which is reminiscent of fractional Sobolev spaces.
In the nonlocal case, however, the cut-and-paste arguments are more delicate,
exactly in view of these interactions coming
``from far away'', so they require
the ``clean point'' procedure introduced in Section~\ref{CLEAN:SECT}.
This procedure is designed
exactly in order to make the remote interactions
sufficiently small in the glueing methods: roughly speaking,
when a glueing procedure makes a sharp angle, the nonlocal
energy increases considerably (that is, it is much more than just the sum
of the contributions to the left and
to the right of the angle). On the other hand, when the function is very flat
in a very large neighborhood of the glueing point,
this additional energy is rather small, because the values 
of the function near this point are basically constant and so they give almost
no contribution to the interaction energy. An additional energy contribution
comes from outside this flatness interval, but, thanks to the decay of the kernel
at infinity, it becomes very small as the flatness interval becomes very large.

In this sense, the notation introduced
in Section~\ref{CLEAN:SECT} aims to give a precise quantification of the procedure
discussed above, also with respect to the energy functional
related to Theorem~\ref{TH:MAIN}.

\section{Toolbox}\label{SE1}

This section collects some auxiliary lemmata needed for
the proofs of the main theorem. 
An ancillary tool for these results is the basic theory of the
fractional Sobolev spaces. In our setting, given
an interval~$J\subseteq \R$, we will consider the so-called
Gagliardo seminorm of a measurable function~$Q:\R\to\R^n$, given by
$$ [Q]_{H^s(J)}:=\left(
(1-s)\,
\iint_{J\times J} \frac{|Q(x)-Q(y)|^2}{|x-y|^{1+2s}}\,dx\,dy \right)^{\frac12}$$
and the complete fractional norm, given by
$$ \|Q\|_{H^s(J)}:=[Q]_{H^s(J)} +\|Q\|_{L^2(J)}.$$
We also denote by~$|J|$ the length of the interval~$J$.
It is useful to observe that~$E(Q)$ controls the Gagliardo seminorm, namely,
by~\eqref{KERNEL},
\begin{equation}\label{CONTROL}
\begin{split}
{\mbox{ if $|J|\le\rho_0$ then }}
\qquad E(Q)\;&\ge 
\iint_{J\times J} K(x-y)\,\big|Q(x)-Q(y) \big|^2\,dx\,dy
\\&\ge \iint_{J\times J}
\frac{\theta_0 \;(1-s)\;\big|Q(x)-Q(y) \big|^2}{|x-y|^{1+2s}}
\,dx\,dy= \theta_0\,[Q]_{H^s(J)}^2\\
{\mbox{ and so }}
\qquad \|Q\|_{H^s(J)}\;&\le \big(\theta_0^{-1} E(Q)\big)^{\frac12}+\|Q\|_{L^\infty(J)}.
\end{split}\end{equation}
In this framework, we recall
a H\"older embedding result 
that is
uniform as~$s\to1$:

\begin{lemma}\label{SOB}
Let~$s_0\in\left(\frac12, 1\right)$ and~$s\in[s_0,1)$.
Let~$J\subset\R$ be an interval of length~$1$. Then,
there exists~$ S_0>0 $, possibly depending on~$n$ and~$s_0$,
such that for any~$Q:J\to\R^n$ we have that
\begin{equation}\label{LA:99:001}
[Q]_{C^{0, s-\frac12 }(J)}\le S_0\,[Q]_{H^s(J)}.\end{equation}
\end{lemma}

The proof of Lemma~\ref{SOB} follows the classical ideas of~\cite{MR0156188}
and can be found essentially in many textbooks.
In any case, since we need here to check that the constants
are uniform in~$s\in[ s_0, 1)$
(recall the footnote on page~\ref{FOO:678uhauu})
and this detail
is often omitted in the existing literature,
for completeness we give a selfcontained proof of Lemma~\ref{SOB}
in Appendix~\ref{CAMPA}.

Now we define the energy functional
\begin{equation}\label{ENE}
I(Q) := E(Q)
+\int_\R a(x)\,W(Q(x))\,dx,\end{equation}
where~$E(Q)$ is the ``free energy'' introduced in~\eqref{def E:Q}.

In the next result
we compute how much the energy charges ``long'' trajectories:

\begin{lemma}\label{FUrbJAK}
Let~$\zeta=(\zeta_1,\dots,\zeta_n)\in\Z^n$, $x_0\in\R$ and~$Q=(Q_1,\dots,Q_n):
\R\to\R^n$ be a measurable function
such that~$Q(x)\in \overline{B_r(\zeta)}$ for any~$x\le x_0$.
Assume that~$I(Q)<+\infty$ and
\begin{equation}\label{895A}
\sup_{x\in\R} |Q_i(x)-\zeta_i|\ge \nu,\end{equation}
for some~$\nu\in\N$, $\nu\ge1$ and~$i\in\{1,\dots,n\}$. Then
$$ I(Q)\ge E(Q)+2\ell_{Q}\,\underline{a}\,\nu
\inf_{ {\rm dist}\, (\tau,\; \Z^n) \ge 1/4 } W(\tau),$$
where~$r$ and~$\underline{a}$ are as in~\eqref{GROW} and~\eqref{GROW:2},
and\begin{equation} \label{ELL} \ell_{Q}:=
\min \left\{ \frac{\rho_0}{2}, \;
\left( \frac{1}{ 4S_0\,\big(\theta_0^{-1} E(Q)\big)^{\frac12}
}\right)^{\frac{2}{2s-1}}\right\}.
\end{equation}
\end{lemma}

\begin{proof} Up to reordering the components of~$Q$, 
we may suppose that~$i=1$. Also, by a translation, we may assume that~$\zeta=0$.

By~\eqref{CONTROL}, we find that~$[Q]_{H^s(J)}\le 
\big(\theta_0^{-1} E(Q)\big)^{\frac12}$, for any interval~$J$ with~$|J|\le\rho_0$.
Consequently,
by scaling Lemma~\ref{SOB}, we obtain that~$[Q]_{C^{0,s-\frac{1}{2}}(J)}$
is bounded by~$S_0\,\big(\theta_0^{-1} E(Q)\big)^{\frac12}$
for any interval~$J$ with~$|J|\le\rho_0$.

In particular, $|Q_1|$ is a continuous curve, which, by~\eqref{895A},
connects~$0$ with~$\nu$ and so it passes through all the points
of the form~$\frac12+m$, for any~$m\in \{0,\dots,\nu-1\}$.
More explicitly, we can say that there exists~$X_m$ such that~$|Q_1(X_m)|=
\frac12+m$, for all~$m\in \{0,\dots,\nu-1\}$.
This says that
\begin{equation}\label{YUAJ}
Q_1(X_m)\in \frac12+\Z.\end{equation}
Let now~$\ell_{Q}$
be as in~\eqref{ELL}.
Then, for any~$x\in [X_m-\ell_{Q},X_m+\ell_{Q}]$,
$$ |Q_1(x)-Q_1(X_m)| \le S_0\,\big(\theta_0^{-1} E(Q)\big)^{\frac12}\,
\ell_{Q}^{s-\frac12}\le \frac14,$$
and so, by~\eqref{YUAJ},
$$ {\rm dist}\, \big(Q_1(x),\; \frac12+\Z\big) \le \frac14,$$
which gives that
$$ {\rm dist}\, \big(Q_1(x),\; \Z\big) \ge \frac14\ge r,$$
for any~$x\in [X_m-\ell_{Q},X_m+\ell_{Q}]$.
Thus, writing~$\tau=(\tau_1,\dots,\tau_n)$ and recalling~\eqref{ZERI di W},
$$ W(Q(x))\ge \inf_{ {\rm dist}\, (\tau_1,\; \Z) \ge 1/4 } W(\tau),$$
for any~$x\in [X_m-\ell_{Q},X_m+\ell_{Q}]$. As a consequence,
\begin{equation*}\begin{split} I(Q)&\ge E(Q)+\sum_{m=0}^{\nu-1} \int_{X_m-\ell_{Q}}^{X_m+\ell_{Q}} a(x)\,W(Q(x))\,dx\\&
\ge E(Q)+ 2\ell_{Q}\,\underline{a}\,\nu 
\inf_{ {\rm dist}\, (\tau_1,\; \Z) \ge 1/4 } W(\tau)\\&\ge
E(Q)+2\ell_{Q}\,\underline{a}\,\nu
\inf_{ {\rm dist}\, (\tau,\; \Z^n) \ge 1/4 } W(\tau)
,\end{split}\end{equation*}
as desired.
\end{proof}

\section{A bit of regularity theory}\label{SeR}

Goal of this section is to establish the following
regularity result for solutions of~\eqref{EQUAZ} that are close to an integer
in large intervals,
with uniform estimates as~$s\to1$:

\begin{lemma}\label{9jkdJJKA}
Let~$s_0\in\left(\frac12,1\right)$
and~$s\in[s_0,1)$.

Let~$T>32$, $\rho>0$, $M_o>0$, $\zeta\in\Z^n$.
Let~$Q\in L^\infty(\R,\R^n)$ be a solution 
of
$${\mathcal{L}}(Q)(x)+a(x) \nabla W(Q(x))=0$$
in~$[-2T,2T]$, with~$E(Q)+\|Q\|_{L^\infty(\R,\R^n)}\le M_o$.

Suppose that
\begin{equation}\label{097ihkGHJHG:00}
{\mbox{$Q(x)\in\overline{B_\rho(\zeta)}$ for any~$x\in
[-2T,2T]$. 
}}\end{equation}
Then
$$ \|Q\|_{C^{0,1}([-T/16,T/16])}\le
\frac{CM_o\,(1-s)}{T^{2s}}+
C\rho,$$
with~$C>0$ depending on~$n$, $s_0$ and on the structural constants
of the kernel and the potential.
\end{lemma}

\begin{proof} Up to a translation, we assume that~$\zeta=0$,
hence~\eqref{097ihkGHJHG:00} becomes
\begin{equation}\label{097ihkGHJHG}
{\mbox{$|Q(x)|\le\rho$ for any~$x\in
[-2T,2T]$.
}}\end{equation}
We let~$\tau_o\in C^\infty_0([-1,1],[0,1])$ be
such that~$\tau_o(x)=1$ for any~$x\in \left[ -\frac12,\frac12\right]$.
We define~$\tau(x):= \tau_o(x/T)$ and~$u(x):=\tau(x)\, Q(x)$.
Notice that, by~\eqref{097ihkGHJHG},
\begin{equation}\label{9iokdf67896pp}
{\mbox{$|u(x)|\le\rho$ for any~$x\in
\R$.}}
\end{equation}
By Lemma~\ref{SOB}, we already know that~$Q$
is continuous
and so it is also a viscosity
solution. Therefore (see e.g. formula~(2.11) in~\cite{MR3211862}),
we have that, in the viscosity sense,
\begin{equation}\label{9ioksc82173eudhc}
\begin{split}
{\mathcal{L}}(u) \;&= \tau\,{\mathcal{L}}(Q)+
Q\,{\mathcal{L}}(\tau)
-B(Q,\tau)\\
&= -\tau \,a\,\nabla W(Q)+
Q\,{\mathcal{L}}(\tau)
-B(Q,\tau)
\end{split}
\end{equation}
in~$[-T,T]$, where
$$ B(Q,\tau)(x):=\int_{\R} K(x-y)\,
\big(Q(x)-Q(y)\big)\big(\tau(x)-\tau(y)\big)\,dy.$$
We use~\eqref{KERNEL}
and we notice that, 
for any~$x\in\left[ -\frac{T}4,\frac{T}4\right]$,
\begin{equation}\label{SI097:1}\begin{split}
|B(Q,\tau)(x)|\;&=
\left|
\int_{\R\setminus [-T/2,T/2]} K(x-y)\,
\big(Q(x)-Q(y)\big)\big(\tau(x)-\tau(y)\big)\,dy
\right|\\
&\le 2M_o\,\Theta_0\,(1-s)\,
\int_{\R\setminus [-T/2,T/2]} \frac{
\big|\tau(x)-\tau(y)\big|}{|x-y|^{1+2s}}\,dy\\
&=\frac{2M_o\,\Theta_0\,(1-s)}{T^{2s}}\,
\int_{\R\setminus [-1/2,1/2]} \frac{
\big|\tau_o(T^{-1}x)-\tau_o(y)\big|}{|T^{-1}x-y|^{1+2s}}\,dy
\\ &\le \frac{CM_o\,\Theta_0\,(1-s)}{T^{2s}},\end{split}\end{equation}
for some~$C>0$.

Furthermore
\begin{eqnarray*}&&\int_\R
\frac{|\tau(x+y)+\tau(x-y)-2\tau(x)|}{|y|^{1+2s}}\,dy
\\ &&\qquad=\frac{1}{T^{2s}}\int_\R
\frac{|\tau_o(T^{-1}x+y)+\tau_o(T^{-1}x-y)-2\tau_o(T^{-1}x)|
}{|y|^{1+2s}}\,dy\le \frac{C}{T^{2s}}\end{eqnarray*}
hence
\begin{equation}\label{SI097:2} 
|Q\,{\mathcal{L}}(\tau)|\le 
\frac{CM_o\,\Theta_0\,(1-s)}{T^{2s}},\end{equation}
up to renaming~$C>0$.

Also, we observe that~$\nabla W$ vanishes in~$\Z^n$, thanks to~\eqref{ZERI di W}.
Thus, if we use~\eqref{GROW}, \eqref{GROW:2} and~\eqref{097ihkGHJHG},
we see that if~$x\in[-2T,2T]$
\begin{equation}\label{SI097:3}
\big|\tau(x) \,a(x)\,\nabla W(Q(x))\big|\le
\overline{a} \big|\nabla W(Q(x))-\nabla W(0)\big|
\le \overline{a}\, \|W\|_{C^{1,1}(\R^n)}\,|Q(x)|
\le C\rho,
\end{equation}
up to renaming~$C$.

So we define
$$ f:=- \tau \,a\,\nabla W(Q)+
Q\,{\mathcal{L}}(\tau)
-B(Q,\tau)$$
and we deduce from~\eqref{SI097:1}, \eqref{SI097:2} and~\eqref{SI097:3}
that
\begin{equation}\label{9iokdf67896pp:2}
\|f\|_{L^\infty([-T/4,T/4],\R^n)}\le 
\frac{CM_o\,\Theta_0\,(1-s)}{T^{2s}}+
C\rho,\end{equation}
up to renaming~$C$. In addition, by~\eqref{9ioksc82173eudhc},
we know that
\begin{equation}\label{68jJasd3J}
{\mathcal{L}}(u)=f\end{equation}
in the sense of viscosity.
So, we consider
any interval~$J$ of length~$1$ contained
in~$\left[-\frac{T}{8},\frac{T}{8}\right]$,
and we denote by~$J'$
the dilation of~$J$
by a factor~$1/2$ with respect to the center of the interval.
Thanks to~\eqref{KERNEL} and~\eqref{KERNEL:2}, we can
use Theorem~61 of~\cite{byapp} for the equation
in~\eqref{68jJasd3J} and obtain that
$$ \|u\|_{C^{0,1}(J')}\le
C\,\big(\|u\|_{L^\infty(\R,\R^n)}+\|f\|_{L^\infty(J,\R^n)}\big).$$
{F}rom this, \eqref{9iokdf67896pp} and~\eqref{9iokdf67896pp:2},
we obtain
$$ \|u\|_{C^{0,1}(J')}\le
\frac{CM_o\,\Theta_0\,(1-s)}{T^{2s}}+
C\rho,$$
up to renaming constants, which gives the desired
result.
\end{proof}

\section{Nonlocal glueing arguments}\label{Se2}

In the classical case, it is rather standard to glue
Sobolev functions that meet at a point. In the fractional setting
this operation is more complicated, since the nonlocal interactions
may increase the energy of the resulting functions.
We will provide in the forthcoming Proposition~\ref{GLUE}
a suitable result which will allow us to use glueing methods.

As a technical point, we remark that we will obtain
in these computations very explicit constants
(in particular,
we check the independence of the constants
from~$s$ as~$s$ is close to~$1$).

We first recall a detailed integrability result of classical
flavor (with technical and conceptual differences
in our cases; similar results in a more classical
framework can be found, for instance,
in Chapter~3 of~\cite{mclean}):

\begin{lemma}\label{LAKK}
Let~$\beta\in(0,+\infty)$. Let~$Q:[0,+\infty)\to\R^n$ be a measurable function such
that
$$ [Q]_{H^s([0,1))}<+\infty\; {\mbox{ and }}\; Q(0)=0.$$
Then
\begin{equation}\label{HAG:AAK}\begin{split}&
\int_0^{+\infty} x^{-2s}\,|Q(x)|^2\,dx\\ &\qquad\le C_{s}\,\left[
\,\int_{0}^{\beta} \left[\int_0^x
\frac{|Q(x)-Q(y)|^2}{|x-y|^{1+2s}}\,dy\right]\,dx +
\frac{ 2\|Q\|_{L^\infty((0,+\infty),\R^n)} }{(2s-1)\,\beta^{2s-1} }\right]
,\end{split}\end{equation}
where
\begin{equation}\label{IKAYY}
C_{s}:=2\left(1+
\frac{4}{(2s-1)^2}\right).\end{equation}
\end{lemma}

For the facility of the reader,
we give the proof of Lemma~\ref{LAKK}
in Appendix~\ref{MATTEO}.

\begin{remark}\label{87uxsII}{\rm
If one formally takes~$\beta=+\infty$
in Lemma~\ref{LAKK}, then~\eqref{HAG:AAK} reads simply
$$ (1-s)\,\int_0^{+\infty} x^{-2s}\,|Q(x)|^2\,dx\le C_{s}\,
[Q]_{H^s([0,+\infty)}^2.$$}\end{remark}

Following is the nonlocal glueing result which fits for our purposes:

\begin{proposition}\label{GLUE}
Let~$T_1\in \R\cup\{-\infty\}$ and~$T_2\in (T_1,+\infty]$.
Let~$x_0\in (T_1,T_2)$ and
$$ \beta \in \big( 0,\;\min\{ T_2-x_0,\; x_0-T_1\}\big].$$
Let~$L:(T_1,x_0]\to\R^n$ 
and~$R:[x_0,T_2)\to\R^n$ be measurable functions
with
\begin{equation}\label{HAGYY00871234}\begin{split}
&\iint_{(T_1,x_0)^2} K(x-y)\,|L(x)-L(y)|^2\,dx\,dy <+\infty\\
{\mbox{and }}\;&
\iint_{(x_0,T_2)^2} K(x-y)\,|R(x)-R(y)|^2\,dx\,dy <+\infty.\end{split}\end{equation}
Assume that~$L(x_0)=R(x_0)$, and let
$$ V(x):=\left\{\begin{matrix}
L(x) & {\mbox{ if }} x\in(T_1, x_0],\\
R(x) & {\mbox{ if }} x\in(x_0,T_2).
\end{matrix}
\right. $$
Then
\begin{equation}\label{7ujsIIIIAA}
\begin{split}
&\iint_{(T_1,T_2)^2} K(x-y)\,|V(x)-V(y)|^2\,dx\,dy \\ \le\;&
\iint_{(T_1,x_0)^2} K(x-y)\,|L(x)-L(y)|^2\,dx\,dy +
\iint_{(x_0,T_2)^2} K(x-y)\,|R(x)-R(y)|^2\,dx\,dy\\
+\;&\tilde{C}_{s}\,(1-s)\,\left[
\int_{x_0-\beta}^{x_0}\left(\int_{x}^{x_0} \frac{|L(x)-L(y)|^2}{|x-y|^{1+2s}}\,dy\right)\,dx +
\int_{x_0}^{x_0+\beta}\left(
\int_{x_0}^x \frac{|R(x)-R(y)|^2}{|x-y|^{1+2s}}\,dy\right)\,dx \right]\\
+\;&
\frac{\hat{C}_{s}\,(1-s)}{\beta^{2s-1}}\,\Big[
\|L\|_{L^\infty((T_1,x_0),\R^n)} +
\|R\|_{L^\infty((x_0,T_2),\R^n)}\Big],
\end{split}\end{equation}
where 
$$ \tilde{C}_{s}:=\frac{2\Theta_0\,C_{s}}{s}
\;{\mbox{ and }}\;
\hat{C}_{s}:=\frac{4\,\Theta_0\,C_{s}}{s\,(2s-1)},$$
and~$C_{s}$ is given in~\eqref{IKAYY}.
\end{proposition}

\begin{remark} {\rm In the spirit of Remark~\ref{87uxsII},
we observe that
if one takes~$K(x):=\frac{1-s}{|x|^{1+2s}}$, then
one can formally take~$\theta_0=\Theta_0=1$ and~$\beta=+\infty$, and also~$T_1=-\infty$ and~$T_2=+\infty$,
hence~\eqref{7ujsIIIIAA}
reduces to
\begin{equation}\label{oaOO}
[V]_{H^s(\R)}^2\le (1+\tilde{C}_{s})\,\Big(
[L]_{H^s((-\infty,x_0))}^2+[R]_{H^s((x_0,+\infty))}^2\Big),\end{equation}
with
$$ \tilde{C}_{s}=
\frac{4}{s}\, \left(1+
\frac{4}{(2s-1)^2}\right).$$
We stress that formula~\eqref{7ujsIIIIAA}
is more complicated, but more precise, than~\eqref{oaOO}:
for instance, if one sends~$s\to1$ in~\eqref{7ujsIIIIAA}
for a fixed~$\beta>0$ and then sends~$\beta\to0$,
one recovers the classical Sobolev case of functions in~$H^1((T_1,T_2))$,
namely that
\begin{equation}\label{78JA:AKK} [V]_{H^1((T_1,T_2))}^2\le 
[L]_{H^1((T_1,x_0))}^2+[R]_{H^1((x_0,T_2))}^2.\end{equation}
On the other hand, formula~\eqref{oaOO} in itself cannot recover~\eqref{78JA:AKK},
since it looses a constant.

In our framework, the possibility of having good control on the constants
plays an important role, for example, in the proof of the forthcoming
Proposition~\ref{STICA888}.
}\end{remark}

\begin{proof}[Proof of Proposition~\ref{GLUE}] Up to a translation, we assume 
that~$x_0=0$ and~$L(x_0)=R(x_0)=0$. We also denote~$D^+:=(0,T_2)$
and~$D^-:=(T_1,0)$. If~$T_1\ne -\infty$, we notice that~$L(T_1)$
may be defined by uniform
continuity, thanks to \eqref{HAGYY00871234} and Lemma~\ref{SOB}.
Thus, we can extend~$L(x):=L(T_1)$ for any~$x\le
T_1$. Similarly, if~$T_2\ne+\infty$, we extend~$R(x):=R(T_2)$ for any~$x>T_2$. In this way,
by Lemma~\ref{LAKK},
\begin{eqnarray*} &&\int_{D^-} |x|^{-2s}\,|L(x)|^2\,dx\\ &&\qquad\le 
C_{s}\,\left[ 
\,\iint_{(-\beta,0)\times(x,0)} \frac{|L(x)-L(y)|^2}{|x-y|^{1+2s}}\,dx\,dy +
\frac{ 2\|L\|_{L^\infty(D^-,\R^n)} }{(2s-1)\,\beta^{2s-1} }\right]\\
{\mbox{and }}&&
\int_{D^+} |x|^{-2s}\,|R(x)|^2\,dx
\\ &&\qquad
\le C_{s}\,\left[
\iint_{(0,\beta)\times(0,x)} \frac{|R(x)-R(y)|^2}{|x-y|^{1+2s}}\,dx\,dy +
\frac{ 2\|R\|_{L^\infty(D^+,\R^n)} }{(2s-1)\,\beta^{2s-1} }\right],\end{eqnarray*}
where~$C_{s}$ is given in~\eqref{IKAYY}.
Therefore, decomposing~$(T_1,T_2)$ into
the two intervals~$D^-$ and~$D^+$, and
recalling~\eqref{KERNEL},
\begin{eqnarray*}&&
\iint_{(T_1,T_2)^2} K(x-y)\,|V(x)-V(y)|^2\,dx\,dy\\ && \quad-
\iint_{(D^-)^2} K(x-y)\,|L(x)-L(y)|^2\,dx\,dy-
\iint_{(D^+)^2} K(x-y)\,|R(x)-R(y)|^2\,dx\,dy\\
&=& 2\iint_{D^-\times D^+} K(x-y)\,|L(x)-R(y)|^2\,dx\,dy
\\ &\le& 4
\iint_{D^-\times D^+} K(x-y)\,\Big( |L(x)|^2+|R(y)|^2\Big)\,dx\,dy
\\ &\le& 4\,\Theta_0\,(1-s) \iint_{D^-\times D^+}
\frac{|L(x)|^2+|R(y)|^2}{|x-y|^{1+2s}}\,dx\,dy
\\ &\leq& 
\frac{4\,\Theta_0\,(1-s)}{2s} 
\left[
\int_{D^-} |x|^{-2s} |L(x)|^2\,dx
+ \int_{D^+} |y|^{-2s} |R(y)|^2\,dy\right]\\&\le&
\frac{2\,\Theta_0\,(1-s)\;C_{s}}{s} 
\,\left[
\iint_{(-\beta,0)\times(x,0)} \frac{|L(x)-L(y)|^2}{|x-y|^{1+2s}}\,dx\,dy \right.\\
&&+\left.
\iint_{(0,\beta)\times(0,x)} \frac{|R(x)-R(y)|^2}{|x-y|^{1+2s}}\,dx\,dy +
\frac{ 2\,\|L\|_{L^\infty(D^-,\R^n)} }{(2s-1)\,\beta^{2s-1} }+
\frac{ 2\,\|R\|_{L^\infty(D^+,\R^n)} }{(2s-1)\,\beta^{2s-1} }\right]
\end{eqnarray*}
as desired.
\end{proof}

\section{A notion of clean intervals and clean points}\label{CLEAN:SECT}

In the classical case, a standard tool consists in glueing
together orbits or linear functions. Due to the analysis
performed in Section~\ref{Se2}, we see that the situation
in the nonlocal case is rather different, since the terms
``coming from infinity'' can produce (and do produce, in general)
a nontrivial contribution to the energy.\medskip

To overcome this difficulty, we will need to modify the
classical variational tools concerning the glueing of different orbits
and of orbits and linear functions. Namely, in our case,
we will always perform this glueing at some ``clean points''
that not only produces values of the functions involved close to the integers,
but also that maintains the function close to the integer value
in a suitably large interval. This will allow us to use the
regularity theory in Section~\ref{SeR} to see that the glueing
occurs with ``almost horizontal'' tangent in a large interval
and, consequently, to bound uniformly the nonlocal contributions
arising from the nonlocal glueing procedure discussed in
Section~\ref{Se2}.

Of course, this part is structurally very different from the classical
case and, to this end, we introduce some new terminology.

\begin{definition}\label{DEF:CLEAN} 
Given~$\rho>0$ and~$Q:\R\to\R^n$, we say that an interval~$J\subseteq\R$
is a ``clean interval'' for~$(\rho,Q)$
if~$|J|\ge |\log\rho|$ and there exists~$\zeta\in\Z^n$
such that
$$ \sup_{x\in J}|Q(x)-\zeta|\le\rho.$$
\end{definition}

Of course, the choice of scaling logarithmically
the horizontal length of the interval with respect to the vertical
oscillations in Definition~\ref{DEF:CLEAN}
is for further computational convenience,
and other choices are
also possible (the convenience of this logarithmic choice
will be explained in details in the forthcoming Remark~\ref{DIAMOND}).

\begin{definition}\label{DEF:CLEAN:PT}
If~$J$ is a bounded clean interval for~$(\rho,Q)$,
the center of~$J$ is called a ``clean point'' for~$(\rho,Q)$.
\end{definition}

Any sufficiently long interval contains a clean interval, and thus
a clean point, according to the following result:

\begin{lemma}\label{CLEAN:LEMMA}
Let~$c_0$ and~$r$ be as in~\eqref{GROW}.
Let~$\underline{a}$ be as in~\eqref{GROW:2} and let~$J\subseteq\R$ be an
interval. Let~$Q:\R\to\R^n$, with~$I(Q)\in(0,+\infty)$.
Let~$\rho\in(0,r]$ with
\begin{equation}\label{09ijn67THJ}\left(
\frac{\rho}{2S_0\,\sqrt{\theta_0^{-1}\,E(Q)}}\right)^{\frac{2}{2s-1}}\le|\log\rho|.\end{equation}
Suppose that
\begin{equation}\label{J:ASSU}
|J|\ge 
\frac{\big[1+ 6\, \big( 2S_0\big)^{\frac{2}{2s-1}} 
\big(I(Q)\big)^{\frac{2s}{2s-1}}\big]
\,|\log\rho|}{c_0\,\underline{a}\,\theta_0^{\frac{1}{2s-1}}\,
\rho^{\frac{4s}{2s-1}}}.
\end{equation}
Then there exists
a clean interval for~$(\rho,Q)$ that is contained in~$J$.
\end{lemma}

\begin{proof} Assume, by contradiction, that
\begin{equation}\label{90iuojdc876543rg}
{\mbox{$J$ does not contain
any clean subinterval.}} \end{equation}
By~\eqref{J:ASSU}, the interval~$J$
contains~$N$ disjoint subintervals, say~$J_1,\dots,J_N$,
each of length~$|\log\rho|$, with
\begin{equation}\label{98iiogur76yuh}
N\ge \frac{5\, \big( 2S_0\big)^{\frac{2}{2s-1}} 
\big(I(Q)\big)^{\frac{2s}{2s-1}}}{c_0\,\underline{a}\,\theta_0^{\frac{1}{2s-1}}\,
\rho^{\frac{4s}{2s-1}}}.
\end{equation}
By~\eqref{90iuojdc876543rg}, none of the subintervals~$J_i$
is clean. Hence, for any~$i\in\{1,\dots,N\}$, there exists~$p_i\in J_i$
such that~$Q(p_i)$ stays at distance larger than~$\rho$ from the integer
points. Now, letting
$$ \ell_\rho :=
\left(  \frac{\rho}{2S_0\,\sqrt{\theta_0^{-1}\,E(Q)}}\right)^{\frac{2}{2s-1}}$$
and recalling Lemma~\ref{SOB}, we have that,
for any~$x\in J'_i:=[p_i-\ell_\rho,p_i+\ell_\rho]$,
$$ |Q(x)-Q(p_i)|\le [Q]_{C^{0,s-\frac12}(J_i)} |x-p_i|^{s-\frac12}
\le S_0\,\sqrt{\theta_0^{-1}\,E(Q)}\, \ell_\rho^{s-\frac12}
=\frac{\rho}{2}.$$
Accordingly,~$Q(x)$ stays at distance larger than~$\frac\rho2$
from the integer points, for any~$x\in J'_i$, and so, by~\eqref{GROW},
$$ W(Q(x))\ge \frac{c_0\,\rho^2}{4}.$$
Also, by~\eqref{09ijn67THJ}, at least half of the interval~$J'_i$
lies in~$J_i$, hence
$$ \int_{J_i\cap J_i'} W(Q(x))\,dx\ge \frac{c_0\,\rho^2\,\ell_\rho}{4}.$$
Summing up over~$i=1,\dots,N$, and using that the intervals~$J_i$
are disjoint, we find that
$$ I(Q)\ge \frac{c_0\,\underline{a}\,\rho^2\,\ell_\rho\, N}{4}.$$
This is a contradiction with~\eqref{98iiogur76yuh}
and so it proves the desired result.
\end{proof}

\begin{remark}\label{56edtycgshd2312989019381}{\rm
In our applications, we will make use of
Lemma~\ref{CLEAN:LEMMA} to orbits whose energy is bounded uniformly.
In this way, condition~\eqref{09ijn67THJ}
simply requires~$\rho$ to be small enough
and~\eqref{J:ASSU} reads
$$|J|\ge\frac{C_*\,|\log\rho|}{\rho^{\frac{4s}{2s-1}}},$$
for some~$C_*>0$.}\end{remark}

\section{Minimization arguments}\label{MI99}

In this section, we introduce the variational problem
that we use in the proof of the main results and we discuss
the basic properties of the minimizers.

For this, we fix~$N\in\N$, $N\ge2$, and we fix~$\zeta_1,\dots,\zeta_N\in\Z^n$
and~$b_1,\dots,b_{2N-2}\in\R$. We assume that~$b_{i+1} \ge b_i + 3$ for any~$i\in\{ 1,\dots,2N-3\}$.

We will use the short notation~$\vec\zeta:=(\zeta_1,\dots,\zeta_N)\in\Z^{nN}$
and~$\vec b:=(b_1,\dots,b_{2N-2})\in\R^{2N-2}$. 
Given~$r$ as in~\eqref{GROW}, we also set
\begin{equation}\label{08453647595436}\begin{split}
\Gamma(\vec\zeta,\vec b) :=\;& \Big\{ Q:\R\to\R^n {\mbox{ s.t. $Q$ is measurable,}}\\
& \qquad {\mbox{ $Q(x)\in\overline{ B_r(\zeta_1)}$ for a.e. $x\in (-\infty, b_1]$,}}\\
& \qquad {\mbox{ $Q(x)\in\overline{ B_r(\zeta_i)}$ for a.e. $x\in [b_{2i-2}, b_{2i-1}]$
and~$i\in\{ 2,\dots,N-1\}$,}}\\
& \qquad {\mbox{ $Q(x)\in\overline{ B_r(\zeta_N)}$ for a.e. $x\in [b_{2N-2},+\infty)$}}
\Big\}.\end{split}\end{equation}
Roughly speaking, the set~$\Gamma(\vec\zeta,\vec b)$
contains all the admissible trajectories that link
any integer point in the array~$\vec\zeta$ to the subsequent one,
up to an error smaller than~$r$, and using the array~$\vec b$
to construct appropriate constrain windows, see Figure~\ref{WIND}.

\begin{figure}
    \centering
    \includegraphics[width=16.8cm]{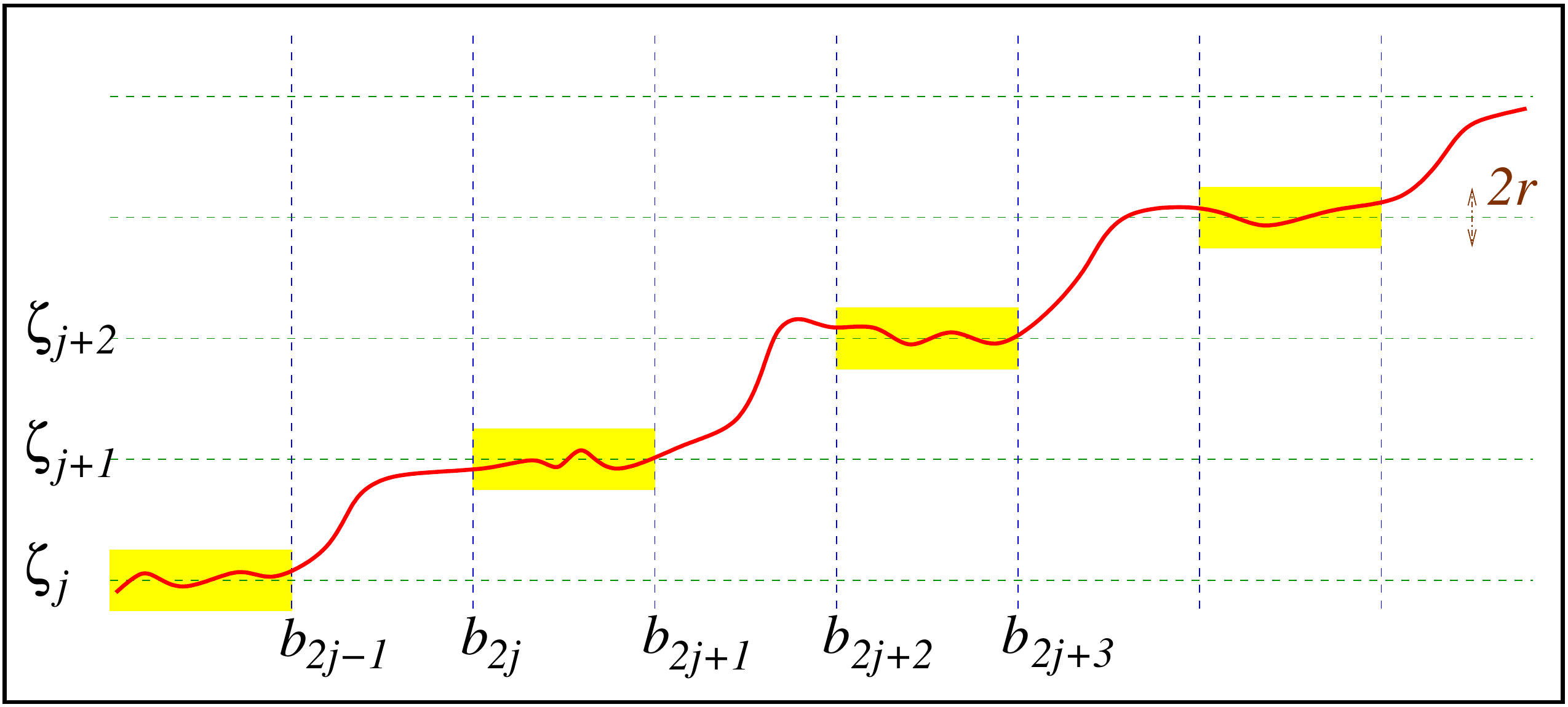}
    \caption{The sets of admissible competitors in~$\Gamma(\vec\zeta,\vec b)$.}
    \label{WIND}
\end{figure}

We also define
$$ M:=\sum_{j=1}^{N-1}|\zeta_{j+1}-\zeta_j|.$$
In this framework, we can consider the minimization problem
of the energy functional introduced in~\eqref{ENE},
according to the following result:

\begin{lemma}\label{HJA:AA} Let~$s_0\in\left(\frac12,1\right)$
and~$s\in[s_0,1)$.
There exists~$Q_*\in \Gamma(\vec\zeta,\vec b)$ such that
\begin{eqnarray}
\label{ASR1} && \sup_{x\in\R} |Q_*(x)-\zeta_1|\le C,\\
\label{ASR2} && I(Q_*)\le C,\\
\label{ASR3} && {\mbox{$[Q_*]_{H^s(J)}\le C$, for any interval~$J$
with~$|J|\le\rho_0$,}}
\\ \label{ORA} && \|Q_* -\zeta_1\|_{C^{0,s-\frac12}(\R)}\le C,
\end{eqnarray}
for some~$C>0$ possibly depending on~$n$, $s_0$, $M$
and
the structural constants 
of the kernel and the potential,
and
\begin{equation}\label{ASR4}
{\mbox{$I(Q_*)\le I(Q)$ for any~$Q\in \Gamma(\vec\zeta,\vec b)$.}}
\end{equation}
In addition,
\begin{equation}\label{ASR5}
\lim_{x\to-\infty} Q_*(x)=\zeta_1\;{\mbox{ and }}\;
\lim_{x\to+\infty} Q_*(x)=\zeta_N.
\end{equation}
\end{lemma}

\begin{proof} Let~$\mu\in C^\infty(\R,\,[0,1/2])$ be such that~$\mu(0)=1/2$
and~$\mu(x)=0$ if~$|x|\ge1$.
Notice that
$$ [1-\mu]_{H^s(\R)}=[\mu]_{H^s(\R)}<+\infty.$$
Let
$$ \eta(x):=\left\{\begin{matrix}
\mu(x) & {\mbox{ if }} x\le 0,\\
1-\mu(x) & {\mbox{ if }} x>0.
\end{matrix}
\right. $$
Notice that~$\eta(x)=0$ if~$x\le -1$
and~$\eta(x)=1$ if~$x\ge1$. Also, by \eqref{oaOO},
$$ [\eta]_{H^s(\R)}^2 \le (1+\tilde{C}_{s})\,\Big(
[\mu]_{H^s( \R)}^2+
[1-\mu]_{H^s( \R)}^2 \Big) = 2(1+\tilde{C}_{s}) \,[\mu]_{H^s(\R)}^2
=:(C'_{s})^2.$$
Let also
$$ \beta_i := \frac{b_{2i-1}+b_{2i}}{2}\qquad{\mbox{
for any }} i\in \{1,\dots,N-1\}$$
and
$$ Q_0(x):=\zeta_1 +\sum_{j=1}^{N-1} (\zeta_{j+1}-\zeta_j)\,
\eta (x-\beta_j).$$
Notice that~$\beta_i$ is an increasing sequence.
We also claim that
\begin{equation}\label{ADM}
Q_0\in \Gamma(\vec\zeta,\vec b).
\end{equation}
To prove this we note that:
\begin{itemize}
\item if~$x\le b_1$ then
$$ x-\beta_j\le b_1-\beta_1 = -\frac{b_2-b_1}{2}\le -\frac{3}{2}$$
for all~$j\in \{1,\dots,N-1\}$, thus~$\eta(x-\beta_j)=0$
for all~$j\in \{1,\dots,N-1\}$, and then~$Q_0(x)=\zeta_1$;
\item if~$i\in\{ 2,\dots,N-1\}$ and~$x\in [b_{2i-2}, b_{2i-1}]$,
then, for all~$j\in \{1,\dots,i-1\}$ we have that
$$ x-\beta_j \ge b_{2i-2} - \beta_{i-1}=\frac{b_{2i-2}-b_{2i-3}}{2}\ge
\frac{3}{2},$$
and thus~$\eta(x-\beta_j)=1$ for all~$j\in \{1,\dots,i-1\}$,
while for all~$j\in \{i,\dots,N-1\}$ we have that
$$ x-\beta_j \le b_{2i-1} - \beta_{i}=-\frac{b_{2i}-b_{2i-1}}{2}\le-
\frac{3}{2},$$
and thus~$\eta(x-\beta_j)=0$ for all~$j\in \{i,\dots,N-1\}$, therefore
a telescopic sum gives that
$$ Q_0(x)=
\zeta_1 +\sum_{j=1}^{i-1} (\zeta_{j+1}-\zeta_j) =
\zeta_1 +(\zeta_{i}-\zeta_1)=\zeta_i;$$
\item if~$x\ge b_{2N-2}$ then
$$ x-\beta_j\ge b_{2N-2}-\beta_{N-1} = \frac{b_{2N-2}-b_{2N-3}}{2}\ge 
\frac{3}{2}$$
for all~$j\in \{1,\dots,N-1\}$, thus~$\eta(x-\beta_j)=1$
for all~$j\in \{1,\dots,N-1\}$, and then a telescopic sum gives that
$$ Q_0(x)=
\zeta_1 +\sum_{j=1}^{N-1} (\zeta_{j+1}-\zeta_j) =
\zeta_1 +(\zeta_{N}-\zeta_1)=\zeta_N.$$
\end{itemize}
These considerations prove~\eqref{ADM}.

Moreover,
$$ [Q_0]_{H^s(\R)} \le \sum_{j=1}^{N-1}|\zeta_{j+1}-\zeta_j|
\,[\eta]_{H^s(\R)} \le C'_{s} \sum_{j=1}^{N-1}|\zeta_{j+1}-\zeta_j|.$$
This and~\eqref{KERNEL} give that
$$ E(Q)\le \Theta_0 \,[Q_0]_{H^s(\R)}^2\le
C'_{s} \Theta_0\,\sum_{j=1}^{N-1}|\zeta_{j+1}-\zeta_j|.$$
Also, we have that~$\eta(x-\beta_j)$ takes integer values 
outside~$[\beta_j-1,\beta_j+1]$
and therefore
$$ \int_\R a(x)\,W(Q_0(x))\,dx\le
\overline{a} \sum_{j=1}^{N-1} \int_{\beta_j-1}^{\beta_j+
1}W(Q_0(x))\,dx \le 2N\overline{a}\,\sup_\R W.$$
Accordingly, we find
\begin{equation}\label{ADMM}
I(Q_0) \le 
C'_{s} \Theta_0\,\sum_{j=1}^{N-1}|\zeta_{j+1}-\zeta_j|
+2N\overline{a}\,\sup_\R W =:C_1.\end{equation}

Now we take a minimizing sequence~$Q_k\in\Gamma(\vec\zeta,\vec b)$, that is
\begin{equation}\label{ADMM:SAR}
\lim_{k\to+\infty} I(Q_k)=\inf_{\Gamma(\vec\zeta,\vec b)} I
\le C_1,\end{equation}
where we also used~\eqref{ADM} and~\eqref{ADMM}.
Then, we write~$\R$ as the disjoint union of intervals
of length~$\rho_0$, say
$$ \R =\bigcup_{\ell\in\N} J_\ell,$$
with~$|J_\ell|=\rho_0$ and it follows from~\eqref{CONTROL}
and \eqref{ADMM:SAR}
that, for any~$\ell\in\N$,
\begin{equation}\label{AD45}
{\mbox{$[Q_k]_{H^s(J_\ell)}$ is bounded independently on~$k$.}}\end{equation}
Also, by~\eqref{ADMM:SAR} and Lemma~\ref{FUrbJAK}, we find that
\begin{equation}\label{AD46}
\sup_{x\in\R} |Q_k(x)-\zeta_1|\le C_2,\end{equation}
for some~$C_2>0$.

By~\eqref{AD45}, \eqref{AD46} and compact embeddings
(see e.g. Theorem~7.1 in~\cite{guida}), and using a diagonal argument,
we obtain that~$Q_k$ converges
a.e. in~$\R$ to some~$Q_*$. By construction, $Q_*\in\Gamma(\vec\zeta,\vec b)$
and, by Fatou Lemma,
$$ \liminf_{k\to+\infty} I(Q_k)\ge I(Q_*).$$
Hence, recalling~\eqref{ADMM:SAR}, we find that~$Q_*$ is the desired
minimizer in~\eqref{ASR4} and that~\eqref{ASR2}
holds true. Then, \eqref{ASR3} follows from~\eqref{CONTROL} and~\eqref{ASR2}.
Moreover, we see that~\eqref{ASR1} is a consequence of~\eqref{AD46},
while~\eqref{ORA} follows from~\eqref{ASR1}, \eqref{ASR3}
and Lemma~\ref{SOB}.

Now we prove~\eqref{ASR5}. We deal with the case of~$x\to+\infty$,
the other case being similar. We argue by contradiction and assume that
there exist~$\alpha_0 >0$ and a sequence~$x_k$ such that~$x_k\to+\infty$
as~$k\to+\infty$ and~$|Q_*(x_k)-\zeta_N|\ge \alpha_0$.
Let~$\ell:=\left(\frac{\alpha_0}{2C}\right)^{\frac{2}{2s-1}}$,
where~$C>0$ is as in~\eqref{ORA}. Then,
by~\eqref{ORA}, we find that, for any~$x\in [x_k-\ell,x_k+\ell]$,
$$ |Q_*(x)-Q_*(x_k)|\le C\,|x-x_k|^{s-\frac12}\le C\,\ell^{\frac{2s-1}{2}}
\le \frac{\alpha_0}{2}$$
and so~$|Q_*(x)-\zeta_N|\ge \frac{\alpha_0}2$
for any~$x\in [x_k-\ell,x_k+\ell]$.
 
Notice also that~$Q_*(x)\in \overline{ B_r(\zeta_N)}$
for any~$x\in [x_k-\ell,x_k+\ell]$, since~$Q_*\in \Gamma(\vec\zeta,\vec b)$,
which says that~$|Q_*(x)-\zeta_N|\in \left[\frac{\alpha_0}2,\,r\right]$.
Therefore, for any~$x\in [x_k-\ell,x_k+\ell]$,
we have that~${\rm dist}(Q_*(x),\Z^n)\ge\alpha_1$, for some~$\alpha_1>0$,
and thus
$$ W(Q_*(x))\ge \inf_{ {\rm dist}(\tau,\Z^n)\ge\alpha_1 } W(\tau).$$
As a consequence
$$ I(Q_*)\ge \underline{a}
\sum_{k=1}^{+\infty} \int_{x_k-\ell}^{x_k+\ell} W(Q_*(x))\,dx\ge
\underline{a}\,\inf_{ {\rm dist}(\tau,\Z^n)\ge\alpha_1 } W(\tau)
\,\sum_{k=1}^{+\infty} (2\ell) =+\infty.$$
This is in contradiction with~\eqref{ASR2} 
and thus we have established~\eqref{ASR5}.
\end{proof}

Now we observe that trajectories
with long excursions have large energy, in a uniform
way, as stated in the following result:

\begin{lemma}\label{LARGE:EN}
Let~$Q\in\Gamma(\vec\zeta,\vec b)$. Assume that
\begin{equation*}
\sup_{x\in\R} |Q_i(x)-\zeta_{1,i}|\ge \nu,\end{equation*}
for some~$\nu\in\N$, $\nu\ge1$ and~$i\in\{1,\dots,n\}$
(where~$\zeta_1=(\zeta_{1,1},\dots,\zeta_{1,n}$).
Then
\begin{equation}\label{TH01GO}
I(Q) \ge \min\left\{ c_1\rho_0\nu,\; \;
\left(\frac{c_1\,c_2}{2s-1}\right)^{\frac{2s-1}{2s}}
\cdot\nu^{\frac{2s-1}{2s} }\right\},\end{equation}
where
$$ c_1:= \underline{a}\,
\inf_{ {\rm dist}\,
(\tau,\; \Z^n) \ge 1/4 } W(\tau)
\qquad\qquad{\mbox{ and }}\qquad\qquad c_2:=
2\,\left( \frac{
\theta_0^{\frac12} }{ 4S_0}\right)^{\frac{2}{2s-1}}
.$$\end{lemma}

\begin{proof} We distinguish two cases.
First, if
$$ \left( \frac{1}{ 4S_0\,\big(\theta_0^{-1} E(Q)\big)^{\frac12}
}\right)^{\frac{2}{2s-1}}\ge \frac{\rho_0}{2},$$
then, recalling~\eqref{ELL}, we see that~$ \ell_{Q}=\rho_0/2$
and so, by Lemma~\ref{FUrbJAK},
$$ I(Q)\ge
\rho_0\,\underline{a}\,\nu
\inf_{ {\rm dist}\, (\tau,\; \Z^n) \ge 1/4 } W(\tau),$$
which implies the desired result in~\eqref{TH01GO}
in this case.

Conversely, if
$$ \left( \frac{1}{ 4S_0\,\big(\theta_0^{-1} E(Q)\big)^{\frac12}
}\right)^{\frac{2}{2s-1}}< \frac{\rho_0}{2},$$
we get from~\eqref{ELL} that
$$  \ell_{Q}=
\left( \frac{1}{ 4S_0\,\big(\theta_0^{-1} E(Q)\big)^{\frac12}
}\right)^{\frac{2}{2s-1}} =
\left( \frac{\theta_0^{\frac12} }{ 
4S_0
}\right)^{\frac{2}{2s-1}} \cdot\frac{1}{(E(Q))^{\frac{1}{2s-1}}}.$$
Hence, in this case, an application of Lemma~\ref{FUrbJAK}
gives that
\begin{equation} \label{0648193jifdi}\begin{split}
I(Q)\ge& \;E(Q)+2
\,\underline{a}\,\nu
\inf_{ {\rm dist}\, (\tau,\; \Z^n) \ge 1/4 } W(\tau)
\left( \frac{\theta_0^{\frac12} }{ 
4S_0
}\right)^{\frac{2}{2s-1}} \cdot\frac{1}{(E(Q))^{\frac{1}{2s-1}}}
\\ =& \; E(Q)+\frac{c_1\,c_2}{ (E(Q))^{\frac{1}{2s-1}}}.
\end{split}\end{equation}
A simple calculus also shows that the function
$$ [0,+\infty) \ni t\;\longmapsto\; 
t+\frac{c_1\,c_2}{ t^{\frac{1}{2s-1}}}$$
takes its minimum at~$t_*=
\left(\frac{c_1\,c_2}{2s-1}
\right)^{\frac{2s-1}{2s}}\cdot\nu^{\frac{2s-1}{2s}}$,
where it attains a value larger than~$t_*$.
Accordingly, 
from~\eqref{0648193jifdi},
$$ I(Q) \ge \left(\frac{c_1\,c_2}{2s-1}
\right)^{\frac{2s-1}{2s}}\cdot\nu^{\frac{2s-1}{2s}},$$
which implies~\eqref{TH01GO} in this case.
\end{proof}

Now we define
$$ J_*:= \bigcup_{i=1}^{N-1} (b_{2i-1},b_{2i})$$
and
\begin{eqnarray*}
L_1 &:=& \big\{ x\in (-\infty,b_1] {\mbox{ s.t. }} |Q(x)-\zeta_1|< r\big\},\\
L_i &:=& \big\{ x\in [b_{2i-2}, b_{2i-1}] 
{\mbox{ s.t. }} |Q(x)-\zeta_i|< r\big\},\qquad{\mbox{ with }}i
\in\{ 2,\dots,N-1\},\\
L_N &:=& \big\{ x\in (b_{2N-2},\infty) {\mbox{ s.t. }} |Q(x)-\zeta_N|< r\big\}.
\end{eqnarray*}
Let also
$$ L:=\bigcup_{ i \in\{ 2,\dots,N-1\} } L_i\qquad{\mbox{
and }}\qquad
F:= J_*\cup L.$$
As usual, by taking inner variations, one sees that
in the set~$F$ the minimization problem is ``free''
and so it satisfies an Euler-Lagrange equation, as stated
explicitly in the next result:

\begin{lemma}\label{HJA:AA:2}
Let~$Q_*$ be as in Lemma~\ref{HJA:AA}.
For any~$x\in F$, we have that
\begin{equation}\label{HJA:AA:2:EQ}
{\mathcal{L}}(Q_*)(x) + a(x)\,\nabla W(Q_*(x)) =0,\end{equation} as defined in~\eqref{U756GA}.
\end{lemma}

\begin{remark}\label{DIAMOND}
{\rm Given an interval~$J\subseteq\R$, 
it is convenient to introduce the notation
\begin{equation}\label{98uiojkmyHGJJJJJ:0}
E_J(Q):= \iint_{J\times J} K(x-y)\,\big|Q(x)-Q(y) \big|^2\,dx\,dy.\end{equation}
For instance, comparing with~\eqref{def E:Q},
we have that~$E_\R=E$. Also, if~$J$ is the disjoint union of~$J_1$ and~$J_2$,
then
\begin{equation*}
E_J(Q) \ge E_{J_1}(Q)+E_{J_2}(Q).
\end{equation*}
With this notation, we are able to glue two functions~$L$
and~$R$ at a point~$x_0$ under the additional 
assumption that
$$ [L]_{C^{0,1}([x_0-\beta,x_0])}\le\eta\quad{\mbox{ and }}\quad
[R]_{C^{0,1}([x_0-\beta,x_0])}\le\eta,$$ for some~$\eta>0$. Indeed,
in this case,
\begin{eqnarray*}&&
\int_{x_0}^{x_0+\beta}\left(
\int_{x_0}^x \frac{|R(x)-R(y)|^2
}{|x-y|^{1+2s}}
\,dy\right)\,dx\le
\eta^2  \int_{x_0}^{x_0+\beta}
\left(
\int_{x_0}^x |x-y|^{1-2s}\,dy\right)\,dx\\
&&\qquad = \frac{\eta^2 \,\beta^{3-2s}}{2\,(3-2s)\,(1-s)},\end{eqnarray*}
and, similarly,
$$ \int_{x_0-\beta}^{x_0}\left(
\int_x^{x_0} \frac{
|L(x)-L(y)|^2}{|x-y|^{1+2s}}\,dy\right)\,dx\le \frac{
\eta^2 \,\beta^{3-2s}}{2\,(3-2s)\,(1-s)}.$$
Therefore, Proposition~\ref{GLUE} gives that
\begin{equation}\label{89ioKKKAIIJHH}
E_{(T_1,T_2)}(V)- E_{(T_1,x_0)}(L)-E_{(x_0,T_2)}(R)
\le C\,\left( \eta^2 \,\beta^{3-2s} + \frac{\|L\|_{L^\infty((T_1,x_0),\R^n)}
+\|R\|_{L^\infty((x_0,T_2),\R^n)} }{\beta^{2s-1}} \right),
\end{equation}
for some~$C>0$.

In particular, one can consider a clean point~$x_0$ (according to
Definitions~\ref{DEF:CLEAN} and~\ref{DEF:CLEAN:PT})
and glue an optimal trajectory~$Q_*$ 
to a linear interpolation with the integer~$\zeta$,
close to~$Q_*(x_0)$, namely consider
$$ V(x):=\left\{
\begin{matrix}
\zeta & {\mbox{ if }} x\le x_0-1,\\
\zeta\,(x_0-x)+Q_*(x_0)\,(x-x_0+1) & {\mbox{ if }}x\in (x_0-1,x_0),\\
Q_*(x) & {\mbox{ if }} x\ge x_0.
\end{matrix}
\right. $$
In this way, and taking~$\rho>0$ suitably small,
by Definitions~\ref{DEF:CLEAN}
and~\ref{DEF:CLEAN:PT}, we know that~$Q_*$ is $\rho$-close to
an integer in~$[x_0-32\beta,x_0+32\beta]$, with
\begin{equation}\label{098idkscf7gihoj777}
\beta=\beta(\rho)=
\frac{|\log\rho|}{32}.\end{equation} In particular, by Lemma~\ref{HJA:AA:2},
we have that~$Q_*$ is solution of~\eqref{EQUAZ} in~$[x_0-32\beta,x_0+32\beta]$.
Also, due to~\eqref{ASR1} 
and~\eqref{ASR2}, both~$\|Q_*\|_{L^\infty(\R,\R^n)}$
and~$I(Q_*)$ are bounded uniformly. Consequently, we can use
Lemma~\ref{9jkdJJKA} with~$T:=16 \beta$ and find that
\begin{equation}\label{098idkscf7gihoj777:BIS}
[Q_*]_{C^{0,1}([x_0-\beta,x_0+\beta ])}\le C\,
\left( \frac{1}{\beta^{2s}}+\rho\right),\end{equation}
up to renaming~$C>0$.

This says that in this case we can take~$\eta:=C\,\left( \frac{1}{\beta^{2s}}+\rho\right)$
and bound the right hand side of~\eqref{89ioKKKAIIJHH} by
\begin{equation} \label{97iodf89789yfghjkHHHHHJ}
C\,\left( \rho^2 \beta^{3-2s}+
\frac{1}{\beta^{3(2s-1)}}+ \frac{1}{\beta^{2s-1}}\right)=\diamondsuit,\end{equation}
thanks to~\eqref{098idkscf7gihoj777},
where we use the notation ``$\diamondsuit$''
to denote quantities that are as small as we wish when~$\rho$
is sufficiently small.

In this way, Proposition~\ref{GLUE} can be used
repeatedly to glue $m$ functions, say~$Q_1,\dots,Q_m$
that are alternatively minimal orbits and linear interpolations
at clean points~$x_1,\dots,x_{m-1}$ where they attach the one
to the other. In this case, if~$Q$ is the function produced
by this glueing procedure, we have that
\begin{equation}\label{DIAMOND:EQ}
\begin{split}
E(Q) \;&\le E_{(-\infty,x_1)}(Q_1)+E_{(x_1,+\infty)}(Q)+\diamondsuit \\
&\le E_{(-\infty,x_1)}(Q_1)+E_{(x_1,x_2)}(Q_2)+
E_{(x_2,+\infty)}(Q)+\diamondsuit\\
&\le E_{(-\infty,x_1)}(Q_1)+E_{(x_1,x_2)}(Q_2)+
E_{(x_2,x_3)}(Q_3)+ E_{(x_3,+\infty)}(Q)+\diamondsuit\\
&\le \dots\le E_{(-\infty,x_1)}(Q_1)+E_{(x_1,x_2)}(Q_2)+
\dots+
E_{(x_{m-2},x_{m-1})}(Q_{m-1})+
E_{(x_{m-1},+\infty)}(Q_m)+\diamondsuit.
\end{split}
\end{equation}
where Proposition~\ref{GLUE} and~\eqref{97iodf89789yfghjkHHHHHJ}
were used repeatedly.
}\end{remark}

\section{Stickiness properties of energy minimizers}\label{MI100}

Now we show that the minimizers have the tendency to
stick at the integers once they arrive sufficiently close to them.
For this, we recall the notation in~\eqref{98uiojkmyHGJJJJJ:0}
and we have:

\begin{proposition}\label{STICA888}
Let~$\rho>0$, $s_0\in\left(\frac12,1\right)$ and~$s\in[s_0,1)$.
Let~$Q_*$ be as in Lemma~\ref{HJA:AA}.

Let~$x_1$, $x_2\in\R$ be clean points for~$(\rho,Q_*)$,
according to Definition~\ref{DEF:CLEAN:PT},
with~$x_2\ge x_1+2$, and
\begin{equation}\label{ALKKA} \max_{i=1,2} 
|Q_*(x_i)-\zeta|\le\rho,\end{equation}
for some~$\zeta\in\Z^n$.

Then
\begin{equation}
\label{STIMA:AA1} E_{(x_1,x_2)}+
\int_{x_1}^{x_2} a(x)\,W(Q_*(x))\,dx\le \diamondsuit
,\end{equation}
with~$\diamondsuit$ as small as we wish if~$\rho$ is suitably small
(the smallness of~$\rho$
depends on~$n$, $s_0$, $M$
and
the structural constants
of the kernel and the potential).

Moreover,
\begin{equation}
\label{STIMA:AA2}
{\mbox{$|Q_*(x)-\zeta|\le r/2$ for every~$x\in [x_1,x_2]$.}}\end{equation}
\end{proposition}

\begin{proof} We define
$$ P(x):= \left\{
\begin{matrix}
Q_*(x) & {\mbox{ if }} x\in (-\infty, x_1),\\
Q_*(x_1)(x_1+1-x)+\zeta(x-x_1) & {\mbox{ if }} x\in [x_1,x_1+1],\\
\zeta & {\mbox{ if }} x\in [x_1+1, x_2-1),\\
Q_*(x_2)(x-x_2+1)+\zeta(x_2-x) & {\mbox{ if }} x\in [x_2-1,x_2],\\
Q_*(x)& {\mbox{ if }} x\in (x_2,+\infty).
\end{matrix}
\right. $$
We observe that, if~$x\in(x_1,x_2)$, then
\begin{equation}\label{89iok66670a0}
\begin{split}
&|P(x)-\zeta|\\
\le\;& \sup_{y\in(x_1,x_1+1)} |Q(x_1)(x_1+1-y)
+\zeta(y-x_1)
-\zeta|+
\sup_{y\in(x_2-1,x_2)} |Q(x_2)(y-x_2-1)
+\zeta(x_2-y)
-\zeta|\\ \le\;&
|Q(x_1)-\zeta|+|Q(x_2)-\zeta|\le 2\rho.\end{split}\end{equation}
We use~\eqref{DIAMOND:EQ} and we obtain that
\begin{equation}\label{97yihokjb6666666:000}
E(P)\le E_{(-\infty,x_1)}(Q_*)+E_{(x_2,+\infty)}(Q_*)
+\diamondsuit\le E(Q_*)-E_{(x_1,x_2)}(Q_*)
+\diamondsuit.\end{equation}
In addition, by~\eqref{GROW} and~\eqref{89iok66670a0}, if~$x\in (x_1,x_2)$
then~$W(P(x))\le 4C_0\rho^2$. Using this and the fact that~$W(P(x))=W(\zeta)=0$
if~$x\in(x_1+1,x_2-1)$, we conclude that
$$ \int_{x_1}^{x_2} W(P(x))\,dx=
\int_{x_1}^{x_1+1} W(P(x))\,dx+\int_{x_2-1}^{x_2} W(P(x))\,dx
\le 8C_0\,\rho^2.$$
Thus, by the minimality of~$Q_*$ and~\eqref{97yihokjb6666666:000},
\begin{eqnarray*}
0&\le&I(P)-I(Q_*)\\ &\le& 
-E_{(x_1,x_2)}(Q_*)-\int_{x_1}^{x_2} a(x)\,W(Q_*(x))\,dx+\diamondsuit,
\end{eqnarray*}
which proves~\eqref{STIMA:AA1}.

Now we prove~\eqref{STIMA:AA2}.
For this,
we assume by contradiction
that there exists~$\tilde x\in[x_1,x_2]$
such that~$|Q_*(\tilde x)-\zeta|> r/2$. 

Since~$Q_*$ is continuous, due to~\eqref{ASR3} and
Lemma~\ref{SOB}, and~$|Q_*(x_1)-\zeta|\le\rho <r/2$,
we obtain that there exists~$\hat x\in[x_1,x_2]$
such that
\begin{equation}\label{9asdwAA}
|Q(\hat x)-\zeta|= \frac{r}2.\end{equation}
More precisely, by~\eqref{ORA},
we know that~$\|Q_* -\zeta_1\|_{C^{0,s-\frac12}(\R)}$
is bounded by a constant~$C_1>1$,
possibly depending on~$n$, $M$
and
the structural constants 
of the kernel and the potential.
In particular, if we define
$$ c_1:= \min\left\{\frac1{10},\;
\left( \frac{r}{4C_1}\right)^{\frac{2}{2s-1}}\right\},$$
we conclude that, for any~$x\in [ \hat x -c_1 , \hat x + c_1]$,
$$ |Q_*(x)-Q_*(\hat x)|\le C_1\,|x-\hat x|^{s-\frac12}\le \frac{r}{4}.$$
This and~\eqref{9asdwAA}
imply that
$$ Q_*(x) \in \overline{ B_{3r/4} (\zeta)\setminus B_{r/4}(\zeta)}$$
and thus
$$ {\rm dist}\,\big(Q_*(x),\Z^n\big)\ge \frac{r}{4},$$
for all~$x\in [ \hat x -c_1, \hat x + c_1]$.
This,~\eqref{ZERI di W} and~\eqref{GROW:2} give that
\begin{equation*}
\int_{\hat x -c_1}^{\hat x + c_1} a(x) W(Q_*(x))\,dx\ge \underline{a}
\int_{\hat x -c_1}^{\hat x + c_1} W(Q_*(x))\,dx \ge
2c_1\,\underline{a} \,
\inf_{ {\rm dist}\, (\tau,\; \Z^n) \ge r/4 } W(\tau)=:c_2.\end{equation*}
Hence, noticing that~$(\hat x-c_1,\hat x+c_1)\subseteq (x_1,x_2)$,
we obtain that
$$ \int_{x_1}^{x_2} a(x) W(Q_*(x))\,dx\ge c_2,$$
and this is in contradiction with~\eqref{STIMA:AA1}
for small~$\rho$. Then, the proof
of~\eqref{STIMA:AA2} is now complete.
\end{proof}

\section{Heteroclinic orbits}\label{JAH:SS1}

Goal of this section is to construct solutions
that emanate from a fixed~$\zeta_1\in\Z^n$ as~$x\to-\infty$
and approach a suitable~$\zeta_2\in\Z^n\setminus\{\zeta_1\}$ as~$x\to+\infty$.
Roughly speaking, this~$\zeta_2$ is chosen to
minimize all the possible energies of the trajectories
connecting two integer points, under the pointwise
constraints considered in Section~\ref{MI99}.

More precisely, fixed~$\zeta_1\ne\zeta_2\in\Z^n$
we consider the minimizer~$Q_*=Q_*^{\zeta_1,\zeta_2}$
as given by Lemma~\ref{HJA:AA}. 

Let
\begin{equation}\label{89JKA:99}
I_{\zeta_1}:=\inf_{\zeta_2\in\Z^n\setminus\{\zeta_1\}} I(Q_*^{\zeta_1,\zeta_2}).\end{equation}
By Lemma~\ref{LARGE:EN} we know that if~$|\zeta_2-\zeta_1|$
is very large, the energy also gets large, therefore
only a finite number of integer points~$\zeta_2$ 
take part to the minimization procedure in~\eqref{89JKA:99}.
Accordingly we can write
\begin{equation}\label{89JKA:99:MI}
I_{\zeta_1}=\min_{\zeta_2\in\Z^n\setminus\{\zeta_1\}} I(Q_*^{\zeta_1,\zeta_2})\end{equation}
and define~${\mathcal{A}}(\zeta_1)$ the family of
all~$\zeta_2
\in\Z^n$ attaining such minimum.\medskip

By construction, ${\mathcal{A}}(\zeta_1)\ne\varnothing$
and contains at most a finite number of elements.
It is interesting to notice that in the case of even potentials~$
{\mathcal{A}}(\zeta_1)$ contains at least two elements:

\begin{lemma}\label{OSS-KA-L1}
Assume that~$W(-\tau)=W(\tau)$ for any~$\tau\in\R^n$. Then,
if~$\zeta_2\in{\mathcal{A}}(\zeta_1)$, also~$2\zeta_1-\zeta_2
\in{\mathcal{A}}(\zeta_1)$.
\end{lemma}

\begin{proof} We observe that
$$ W(2\zeta_1-Q(t))=W(-Q(t))=W(Q(t))$$
in this case, and so the desired claim follows.
\end{proof}

Our goal is now to show that when connecting~$\zeta_1$ to~$\zeta_2\in
{\mathcal{A}}(\zeta_1)$, the optimal trajectory does not get close to
other integer points.
This will be accomplished in the forthcoming Corollary~\ref{UELO:2}.
To this end, we give the following result:

\begin{lemma}\label{UELO}
Let~$s_0\in\left(\frac12,1\right)$ and~$s\in[s_0,1)$.
There exists~$\rho_*>0$,
possibly depending on~$n$, $s_0$ and
the structural constants
of the kernel and the potential,
such that if~$\rho\in(0,\rho_*]$
the following statement holds.

Let~$\tilde\zeta\in\Z^n$
and~$Q\in \Gamma(\zeta_1,\tilde\zeta,b_1,b_2)$.
Assume that there exist~$\zeta\in\Z^n\setminus\{\zeta_1,\tilde\zeta\}$
and a clean point~$x_*\in (b_1,b_2-1)$ for~$Q$ such that~$
Q(x_*)\in\overline{ B_\rho(\zeta)}$.

Assume also that~$Q\in C^{0,\alpha}(\R)$, for some~$\alpha\in(0,1)$, and that
\begin{equation}\label{LI:0988jhvhjakjhdfiwqo}
[Q]_{C^{0,1}(\left[x_*-\frac{|\log\rho|}2,\;x_*+\frac{|\log\rho|}2\right])}\le C\,\left(
\frac{1}{|\log\rho|^{2s}}+\rho
\right)
\end{equation}
for some~$C>0$.
Then there exists~$c>0$, 
depending on~$C$, $\alpha$, $n$ and
the structural constants
of the kernel and the potential, such that
$$ I(Q)\ge I( Q_*^{\zeta_1,\zeta_2} ) + c.$$
\end{lemma}

\begin{proof} 
We define
$$ P(x):=\left\{
\begin{matrix}
Q(x) & {\mbox{ if }} x\le x_*,
\\ Q(x_*)(x_*+1-x)+
\zeta(x-x_*) & {\mbox{ if }} x\in(x_*,x_*+1),\\
\zeta & {\mbox{ if }} x>x_*+1.
\end{matrix}
\right.$$
By construction~$P\in\Gamma(\zeta_1,\zeta,b_1,b_2)$ and~$\zeta\ne\zeta_1$, 
therefore, 
using the minimality of~$Q_*^{\zeta_1,\zeta_2}$, 
\begin{equation}\label{9ujkdYYYIA890654}
I(Q_*^{\zeta_1,\zeta_2})\le I(P).
\end{equation}
On the other hand, using~\eqref{DIAMOND:EQ}, we see that
\begin{equation}\label{98fjJJJJ}
I(P)-I(Q)\le\int_{x_*}^{+\infty} a(x)\,
\Big[ W(P(x))-W(Q(x))\Big]\,dx+\diamondsuit.\end{equation}
Now we use that~$\zeta\ne\tilde\zeta$
and that~$Q(b_2)\in\overline{B_r(\tilde\zeta)}$
to find~$y_*\in [x_*,b_2]$ such that~$Q(y_*)$
stays at distance~$1/4$ from~$\Z^n$. 
Then, by the continuity assumption on~$Q$, 
we find an interval of the form~$[y_*,y_*+\ell']$
such that~$Q(x)$
stays at distance at least~$1/8$ from~$\Z^n$
for all~$x\in [y_*,y_*+\ell']$. Accordingly
$$ \int_{x_*}^{+\infty} a(x)\,W(Q(x)) \,dx\ge
\underline{a}\,\int_{y_*}^{y_*+\ell'} W(Q(x)) \,dx
\ge \underline{a}\,\ell'
\,\inf_{{\rm dist}\,(\tau,\Z^n)\ge1/8} W(\tau)=:\tilde c.$$
Plugging this into~\eqref{98fjJJJJ} and using the definition of~$P$,
we obtain
$$ I(P)-I(Q)\le
\diamondsuit-\tilde c.$$
Thus, we choose~$\rho$ small enough (which gives~$\diamondsuit$
small enough)
and we find 
$$ I(P)-I(Q)\le
-\frac{\tilde c}{2}.$$
This and~\eqref{9ujkdYYYIA890654}
imply the desired result.
\end{proof}

As a consequence of Lemma~\ref{UELO} we obtain:

\begin{corollary}\label{UELO:2}
Let~$s_0\in\left(\frac12,1\right)$ and~$s\in[s_0,1)$.
There exists~$\rho_*>0$,
possibly depending on~$n$ and
the structural constants
of the kernel and the potential,
such that if~$\rho\in(0,\rho_*]$
the following statement holds.

Let~$\zeta_1\in\Z^n$ and~$\zeta_2\in{\mathcal{A}}(\zeta_1)$.
Assume that there exist~$\zeta\in\Z^n$
and a clean point~$x_*\in (b_1,b_2-1)$ such that~$Q_*^{\zeta_1,\zeta_2}(x_*)\in\overline{ B_\rho(\zeta)}$.

Then~$\zeta\in\{\zeta_1,\zeta_2\}$.
\end{corollary}

\begin{proof}
Suppose by contradiction that~$\zeta\not\in\{\zeta_1,\zeta_2\}$.
Then~$Q_*^{\zeta_1,\zeta_2}$ satisfies the assumptions
of Lemma~\ref{UELO} with~$\tilde\zeta:=\zeta_2$
(recall~\eqref{ORA}
in order to fulfill the continuity condition in Lemma~\ref{UELO}, 
and also~\eqref{098idkscf7gihoj777} and~\eqref{098idkscf7gihoj777:BIS}
in order to fulfill the Lipschitz condition in~\eqref{LI:0988jhvhjakjhdfiwqo}).
Hence, using Lemma~\ref{UELO} with~$Q:=Q_*^{\zeta_1,\zeta_2}$,
we obtain that~$ I( Q_*^{\zeta_1,\zeta_2})
\ge I( Q_*^{\zeta_1,\zeta_2} ) + c$, with~$c>0$,
which is an obvious
contradiction.
\end{proof}

Now we are in the position of establishing the existence
of heteroclinic orbits connecting~$\zeta_1\in\Z^n$ and~$\zeta_2\in
{\mathcal{A}}(\zeta_1)$.

\begin{theorem}\label{HET}
Let~$s_0\in\left(\frac12,1\right)$ and~$s\in[s_0,1)$.
Assume that~\eqref{ASS:A:NONDEG} holds.

There exist $\epsilon_*>0$
and~$b_2>b_1\in\R$,
possibly depending on~$n$, $s_0$ and
the structural constants 
of the kernel and the potential, 
such that if~$\epsilon\in(0,\epsilon_*]$,
the following statement holds.

Let~$\zeta_1\in\Z^n$ and~$\zeta_2\in{\mathcal{A}}(\zeta_1)$.

Then~$Q_*^{\zeta_1,\zeta_2}$ is a solution of~\eqref{EQUAZ}.
\end{theorem}

\begin{proof} By~\eqref{ASR2} and Lemma~\ref{LARGE:EN},
we know that~$I(Q_*^{\zeta_1,\zeta_2})$ is bounded by some quantity
(independently on the choice of~$b_1$ and~$b_2$).

We fix~$\rho\in(0,r)$, to be taken sufficiently small
and we define
$$ L:=\frac{\pi}{12\epsilon}.$$
We suppose that~$\epsilon$ is so small that
\begin{equation} \label{L:DEF:pre}
L\ge \frac{C_*\,|\log\rho|}{\rho^{\frac{4s}{2s-1}} },
\end{equation}
for a suitably large constant~$C_*>0$ (of course,
condition~\eqref{L:DEF:pre} is just a smallness condition on~$\epsilon$
and~$C_*>0$ is chosen so that \eqref{J:ASSU} is satisfied).
 
Let also
\begin{equation}\label{89uIJBAAJKAlAK}
b_1:=L\quad{\mbox{ and }}\quad b_2:=23L.\end{equation} 
By~\eqref{ASS:A:NONDEG}
we have, for any~$x\in [b_1-L,b_1+2L]$ (that is~$\epsilon x\in
\left[0,\frac{\pi}{4}\right]$),
\begin{equation}\label{A:EX1}
\begin{split}
&a(x)-a(x+L) = a_2\,
\left[ \cos(\epsilon x)-\cos\left(\epsilon x+\frac{\pi}{12}\right)\right]
\\ &\qquad =a_2\,
\left[ \left( 1-\cos\frac{\pi}{12}\right)
\cos(\epsilon x)+\sin\frac\pi{12}\,\sin(\epsilon x)\right]\ge
a_2\,\left( 1-\cos\frac{\pi}{12}\right)
\cos \frac{\pi}{4} =:\gamma,
\end{split}
\end{equation}
with~$\gamma>0$.

Also, for any~$x\in [b_2-2L, b_2+L]$ (i.e.~$x\in[21L,24L]$)
we define~$\tilde x:=\frac{2\pi}{\epsilon}-x\in [0,3L]=[b_1-L,b_1+2L]$,
and we use the~$\frac{2\pi}{\epsilon}$-periodicity
of~$a$, the fact that~$a$ is even and~\eqref{A:EX1}
to obtain
\begin{equation}\label{A:EX2}
a(x-L)-a(x) = a(-\tilde x-L) -a(-\tilde x)=
a(\tilde x+L) -a(\tilde x)\le -\gamma.
\end{equation}

Now, to prove Theorem~\ref{HET},
we want to show that~$Q_*^{\zeta_1,\zeta_2}$
does not touch the constraints of~$\Gamma(\zeta_1,\zeta_2,b_1,b_2)$,
as given in~\eqref{08453647595436} (then the result would follow from
Lemma~\ref{HJA:AA:2}).

That is, our objective is to show that~$Q_*^{\zeta_1,\zeta_2}(x)$
does not touch~$\partial B_r(\zeta_1)$ when~$x\le b_1$
and does not touch~$\partial B_r(\zeta_2)$ when~$x\ge b_2$.

We assume, by contradiction, that 
\begin{equation}\label{C67HJA}
{\mbox{there exists $x_1\le b_1$ such that~$Q_*^{\zeta_1,\zeta_2}(x_1)\in
\partial B_r(\zeta_1)$,
}}\end{equation}
the other case being similar (just using~\eqref{A:EX2}
in the place of~\eqref{A:EX1}).

By \eqref{ASR5},
there exist sequences~$x_k\le b_1$, with~$x_k
\to-\infty$ as~$k\to+\infty$ and~$y_k\ge b_2$, with~$y_k\to+\infty$
as~$k\to+\infty$, and such that
\begin{equation}\label{142536748KKA}
{\mbox{$Q_*^{\zeta_1,\zeta_2}(x_k)
\in B_\rho(\zeta_1)$ and $Q_*^{\zeta_1,\zeta_2}(y_k)
\in B_\rho(\zeta_2)$.}}\end{equation}
We
observe that
$$ b_2-b_1\ge 3L.$$
Hence, by~\eqref{L:DEF:pre},
condition~\eqref{J:ASSU} is satisfied by
the interval~$(b_1+L,b_1+2L)\subseteq(b_1+L,b_2-L)$ 
(recall Remark~\ref{56edtycgshd2312989019381}).
Consequently, by Lemma~\ref{CLEAN:LEMMA},
\begin{equation}\label{098ujssfs}
\begin{split}
&{\mbox{there exist a clean point~$x_*\in
(b_1+L,b_1+2L)$ and~$\zeta\in\Z^n$}}\\&{\mbox{such that~$Q_*^{\zeta_1,\zeta_2}(x_*)\in 
\overline{B_\rho(\zeta)}$.}}\end{split}\end{equation}
By Corollary~\ref{UELO:2}, we obtain that only two cases may occur,
namely either~$\zeta=\zeta_1$ or~$\zeta=\zeta_2$.

Suppose first that~$\zeta=\zeta_1$. 
Then, in virtue of~\eqref{142536748KKA} and~\eqref{STIMA:AA2} in Proposition~\ref{STICA888},
we have that~$Q_*^{\zeta_1,\zeta_2}(x)\in \overline{B_{r/2}(\zeta_1)}$
for every~$x\in [x_k,x_*]$ and so, by sending~$k\to+\infty$,
for every~$x\in (-\infty,x_*]$. In particular, we get that~$Q_*^{\zeta_1,\zeta_2}(x)\in \overline{B_{r/2}(\zeta_1)}$
for every~$x\le b_1$ and this is in contradiction with~\eqref{C67HJA}.

Therefore, it only remains to check what happens if
\begin{equation}\label{095678jhgF}
\zeta=\zeta_2.\end{equation}
In this case, we use~\eqref{142536748KKA}
and~\eqref{STIMA:AA2} in Proposition~\ref{STICA888} to
see that~$Q_*^{\zeta_1,\zeta_2}(x)\in \overline{B_{r/2}(\zeta_2)}$
for every~$x\in [x_*,y_k]$ and so, in particular,
\begin{equation}\label{98swHHH}
{\mbox{$Q_*^{\zeta_1,\zeta_2}(x)\in \overline{
B_{r/2}(\zeta_2)}$ for every~$x\ge b_2-L$.}}
\end{equation} 
Now we define~$P(x):=Q_*^{\zeta_1,\zeta_2}(x-L)$.
Due to~\eqref{98swHHH}, we have that~$P\in\Gamma(\zeta_1,\zeta_2,b_1,b_2)$
and therefore, by the minimality of $Q_*^{\zeta_1,\zeta_2}$,
\begin{equation}\label{9ytguijh}
\begin{split}
& 0 \le I(P)-I(Q_*^{\zeta_1,\zeta_2})=
\int_\R a(x)\, W(P(x))\,dx -\int_\R a(x)\, W(Q_*^{\zeta_1,\zeta_2}(x))\,dx\\
&\qquad=
\int_\R a(x)\, W(Q_*^{\zeta_1,\zeta_2}(x-L
))\,dx -\int_\R a(x)\, W(Q_*^{\zeta_1,\zeta_2}(x))\,dx
\\ &\qquad=
\int_\R \big[ a(x+L)-a(x) \big]\, W(Q_*^{\zeta_1,\zeta_2}(x))\,dx.
\end{split}\end{equation}
Now, recalling~\eqref{L:DEF:pre}, we see that
condition~\eqref{J:ASSU} is satisfied by
the interval~$(b_1-L,b_1)$ and so, by Lemma~\ref{CLEAN:LEMMA},
we find some~$\zeta_\sharp\in\Z^n$ and
a clean point~$x_{\sharp}\in (b_1-L,b_1)$
with~$Q_*^{\zeta_1,\zeta_2}(x_{\sharp})\in\overline{B_\rho(\zeta_\sharp)}$.
Since~$Q_*^{\zeta_1,\zeta_2}\in \Gamma(\zeta_1,\zeta_2,b_1,b_2)$,
necessarily~$\zeta_\sharp=\zeta_1$.

Accordingly, by~\eqref{STIMA:AA1},
and recalling~\eqref{098ujssfs} and~\eqref{095678jhgF}, for large~$k$ we have that
$$ \int_{x_k}^{x_\sharp} a(x)\,W(Q_*^{\zeta_1,\zeta_2}(x))\,dx\le 
\diamondsuit
\qquad{\mbox{ and }}\qquad
\int_{x_*}^{y_k} a(x)\,W(Q_*^{\zeta_1,\zeta_2}(x))\,dx\le \diamondsuit,$$
and thus, sending~$k\to+\infty$,
$$ \int_{-\infty}^{b_1-L} W(Q_*^{\zeta_1,\zeta_2}(x))\,dx
+\int_{b_1+2L}^{+\infty} W(Q_*^{\zeta_1,\zeta_2}(x))\,dx\le 
\diamondsuit.$$
Using this and~\eqref{A:EX1} into~\eqref{9ytguijh}, we conclude that          
\begin{equation}\label{9ijks78ds7JJ}
\begin{split}
0\;& \le \diamondsuit
+
\int_{b_1-L}^{b_1+2L} \big[ a(x+L)-a(x) \big]\, W(Q_*^{\zeta_1,\zeta_2}(x))\,dx
\\&\le 
\diamondsuit
-\gamma
\int_{b_1-L}^{b_1+2L} W(Q_*^{\zeta_1,\zeta_2}(x))\,dx.\end{split}\end{equation}
Now we observe that~$Q_*^{\zeta_1,\zeta_2}(b_1-L)\in\overline{B_r(\zeta_1)}$
and~$Q_*^{\zeta_1,\zeta_2}(x_*)\in\overline{B_r(\zeta_2)}$,
due to~\eqref{098ujssfs} and~\eqref{095678jhgF}. Therefore, by continuity, there exists~$y_*\in
(b_1-L,x_*)\subseteq(b_1-L,b_1+2L)$ 
such that~$Q_*^{\zeta_1,\zeta_2}(y_*)$ stays at distance~$1/4$
from~$\Z^n$. By~\eqref{ORA}, we find an interval~$J_*$ of uniform length
centered at~$y_*$ such that~$Q_*^{\zeta_1,\zeta_2}(x)$
stays at distance greater than~$1/8$
from~$\Z^n$, for any~$x\in J_*$. So we let~$J_\sharp:= J_*\cap (b_1-L,b_1+2L)$
and we get that~$|J_\sharp|\ge |J_*|/2\ge \tilde c$, for some~$\tilde c>0$,
and
$$ \int_{b_1-L}^{b_1+2L} W(Q_*^{\zeta_1,\zeta_2}(x))\,dx
\ge \int_{J_\sharp} W(Q_*^{\zeta_1,\zeta_2}(x))\,dx\ge\tilde c\,
\inf_{{\rm dist}\,(\tau,\Z^n)\ge1/8} W(\tau)=:\hat{c}.$$
By plugging this into~\eqref{9ijks78ds7JJ}, we conclude that
$$ 0\le 
\diamondsuit-\hat{c}\gamma.$$
The latter quantity is negative for small~$\rho$ and so we have obtained
the desired contradiction.
\end{proof}

\section{Chaotic orbits and proof of Theorem~\ref{TH:MAIN}}\label{JAH:SS2}

This section deals with the construction
of orbits which shadow a given sequence of integer points.
The integers are chosen in such a way that there
is an heteroclinic orbit joining them,
as given by~\eqref{89JKA:99:MI}.

We have seen in Corollary~\ref{UELO:2}
that, when joining two integer points in an optimal way,
it is not worth to get close to other integers.
Now we want to prove a global version of this fact,
namely, when connecting several integer points, in the excursion
between two of them 
it is not worth to get close to other integers.
Of course, the situation in this case
is more complicated than 
the one in Corollary~\ref{UELO:2},
because a single heteroclinic is not a good competitor
for the whole multibump trajectory
(even in the local case, and the nonlocal
feature of the energy gives additional complications
when cutting the orbits).

In this context, the result that we have is the following:

\begin{proposition}\label{UELO:j}
Let~$s_0\in\left(\frac12,1\right)$ and~$s\in[s_0,1)$.
There exist
$\rho_*>0$ and~$C_*>0$,
possibly depending on~$n$, $s_0$ and
the structural constants
of the kernel and the potential,
such that if~$\rho\in(0,\rho_*]$
the following statement holds.

Assume that
\begin{equation}
{\mbox{$\xi_{i+1}\in {\mathcal{A}}(\zeta_i)$ for all~$i\in\{ 1,\dots,N-1\}$}}
\end{equation}
and that
\begin{equation}\label{9ucjvxJJJJJ1616}
b_{i+1}\ge b_i + \frac{C_*\,|\log\rho|}{ 
\rho^{\frac{4s}{2s-1}} }
\;
{\mbox{for all~$i\in\{ 1,\dots,2N-3\}$.}}
\end{equation}
Let~$Q_*\in\Gamma(\vec\zeta,\vec{b})$ be the minimal trajectory
given in Lemma~\ref{HJA:AA}.

Suppose that there exist~$\zeta\in\Z^n$, $j\in\{0,\dots,N-2\}$
and a clean point~$x_*\in [b_{2j+1},b_{2j+2}-1]$ such that
\begin{equation}
Q_*(x_*)\in \overline{ B_\rho(\zeta) }.
\end{equation}
Then~$\zeta\in \{\zeta_{j+1},\zeta_{j+2} \}$.
\end{proposition}

\begin{remark} {\rm When~$N=2$ and~$j=0$,
the claim in Proposition~\ref{UELO:j}
reduces to that in Corollary~\ref{UELO:2}.
}\end{remark}

\begin{proof}[Proof of Proposition~\ref{UELO:j}]
The idea is, roughly speaking, that we can
diminish the energy by glueing a heteroclinic
in lieu of the wide excursion.
The argument is depicted in Figure~\ref{GLUE-FIG} and
the rigorous, and not trivial, details are the following.

\begin{figure}
    \centering
    \includegraphics[width=16.8cm]{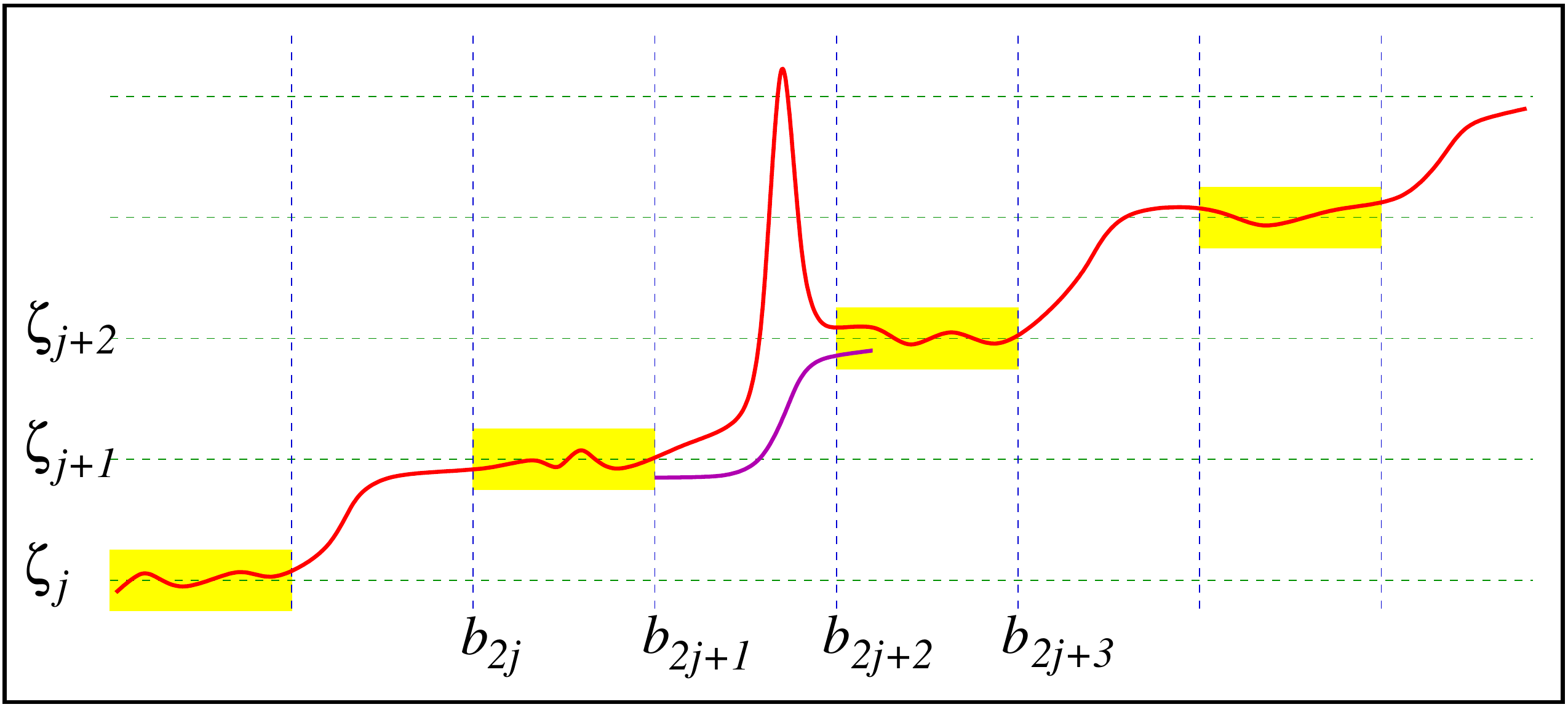}
    \caption{Glueing~$Q_*$ with~$Q_*^{\zeta_{j+1},\zeta_{j+2}}$.}
    \label{GLUE-FIG}
\end{figure}

We argue by contradiction and we suppose that 
\begin{equation}\label{HJAKLL}
\zeta\not\in \{\zeta_{j+1},\zeta_{j+2} \}.\end{equation}
Thanks to~\eqref{9ucjvxJJJJJ1616},
we can exploit Lemma~\ref{CLEAN:LEMMA}
and find clean points for~$Q_*^{\zeta_{j+1},\zeta_{j+2}}$, namely
\begin{eqnarray*} &&y_{*,1}\in 
(b_{2j+1}-C \rho^{-\frac{4s}{2s-1}}\,|\log\rho|,\;b_{2j+1}-1)\\
{\mbox{and }}&& y_{*,2}\in (b_{2j+2}+1,\;b_{2j+2}+C
\rho^{-\frac{4s}{2s-1}}\,|\log\rho| )\end{eqnarray*} such that
\begin{eqnarray*}
&&\sup_{x\in \left[y_{*,1}-\frac{|\log\rho|}{2},\;
y_{*,1}+\frac{|\log\rho|}{2}\right]}|Q_*^{\zeta_{j+1},\zeta_{j+2}}(x)-\zeta_{j+1}|\le \rho\\
{\mbox{and }}
&&\sup_{x\in \left[y_{*,2}-\frac{|\log\rho|}{2},\;
y_{*,2}+\frac{|\log\rho|}{2}\right]}
|Q_*^{\zeta_{j+1},\zeta_{j+2}}(x)-\zeta_{j+2}|\le \rho.\end{eqnarray*}
Similarly, we find clean points for~$Q_*$, say
\begin{eqnarray*} &&
z_{*,1}\in
(b_{2j},\;b_{2j}+C \rho^{-\frac{4s}{2s-1}}\,|\log\rho|)
\\ {\mbox{and }}&&z_{*,2}\in
(b_{2j+3}-C\rho^{-\frac{4s}{2s-1}}\,|\log\rho|,\;b_{2j+3})\end{eqnarray*}
with
\begin{eqnarray*}
&&\sup_{x\in \left[z_{*,1}-\frac{|\log\rho|}{2},\;
z_{*,1}+\frac{|\log\rho|}{2}\right]}
|Q_*(x)-\zeta_{j+1}|\le \rho\\ {\mbox{ and }}&&
\sup_{x\in \left[z_{*,2}-\frac{|\log\rho|}{2},\;
z_{*,2}+\frac{|\log\rho|}{2}\right]}
|Q_*(x)-\zeta_{j+2}|\le \rho.
\end{eqnarray*} 
Then we define
$$ Q^\sharp(x):=\left\{
\begin{matrix}
\zeta_{j+1} & {\mbox{ if }} x< z_{*,1}-1,\\
Q_*(z_{*,1})\,(x-z_{*,1}+1)+\zeta_{j+1}\,(z_{*,1}-x) 
& {\mbox{ if }} x\in[z_{*,1}-1,z_{*,1}],\\
Q_*(x) & {\mbox{ if }} x\in(z_{*,1},z_{*,2}),\\
Q_*(z_{*,2})\,(z_{*,2}+1-x)
+\zeta_{j+2}\,(x-z_{*,2})
& {\mbox{ if }} x\in[z_{*,2},z_{*,2}+1],\\
\zeta_{j+2} & {\mbox{ if }} x> z_{*,2}+1.
\end{matrix}
\right. $$
Thus, recalling the notation in Remark~\ref{DIAMOND}
and formula~\eqref{DIAMOND:EQ},
\begin{equation}\label{8ujkdHHHH1728394}
E(Q^\sharp) \le
E_{(z_{*,1},z_{*,2})}(Q_*) + \diamondsuit.\end{equation}
On the other hand, by construction~$x_*\in (z_{*,1},z_{*,2})$, therefore
\begin{equation}\label{9iujdUUUU}
Q^\sharp(x_*)
=Q_*(x_*)\in\overline{B_\rho(\zeta)}.\end{equation}
Notice also that~$Q^\sharp\in\Gamma(\zeta_{j+1},\zeta_{j+2},b_{2j+1},b_{2j+2})$.
Hence, we use~\eqref{HJAKLL} and~\eqref{9iujdUUUU}
in combination with Lemma~\ref{UELO}, to find that
\begin{equation*}
I(Q^\sharp)\ge I( Q_*^{\zeta_{j+1},\zeta_{j+2}}) + c,
\end{equation*}
for some~$c>0$. This and~\eqref{8ujkdHHHH1728394} give that
\begin{equation}\label{980iodwe6789}
\begin{split}
c\;&\le I(Q^\sharp)-I(Q_*^{\zeta_{j+1},\zeta_{j+2}}) \\
&\le E_{(z_{*,1},z_{*,2})}(Q_*) -E(Q_*^{\zeta_{j+1},\zeta_{j+2}})
+\int_{z_{*,1}}^{z_{*,2}} a(x)\,W(Q_*(x))\,dx
-\int_\R a(x)\,W(Q_*^{\zeta_{j+1},\zeta_{j+2}}(x))\,dx
+ \diamondsuit\\
&\le E_{(z_{*,1},z_{*,2})}(Q_*) -E_{(z_{*,1},z_{*,2})}(Q_*^{\zeta_{j+1},\zeta_{j+2}})
+\int_{z_{*,1}}^{z_{*,2}} a(x)\,\big[W(Q_*(x))
-W(Q_*^{\zeta_{j+1},\zeta_{j+2}}(x))\big]\,dx
+ \diamondsuit
.\end{split}
\end{equation}
Now we define
$$ \tilde{Q}(x):=\left\{
\begin{matrix}
Q_*(x) & {\mbox{ if }} x< z_{*,1},\\
Q_*(z_{*,1})\,(z_{*,1}+1-x)+ \zeta_{j+1}\,(x-z_{*,1})
& {\mbox{ if }} x\in [z_{*,1},z_{*,1}+1],\\
\zeta_{j+1} & {\mbox{ if }} x\in(z_{*,1}+1,\; y_{*,1}-1),\\
Q_*^{\zeta_{j+1},\zeta_{j+2}}(y_{*,1})\,
(x-y_{*,1}+1)+\zeta_{j+1}\,(y_{*,1}-x)
& {\mbox{ if }} x\in[y_{*,1}-1,y_{*,1}],\\
Q_*^{\zeta_{j+1},\zeta_{j+2}}(x) & {\mbox{ if }} x\in(y_{*,1},y_{*,2}),\\
Q_*^{\zeta_{j+1},\zeta_{j+2}}(y_{*,2})\,(y_{*,2}+1-x)
+\zeta_{j+2}\,(x-y_{*,2})
& {\mbox{ if }} x\in[y_{*,2},\;y_{*,2}+1],\\
\zeta_{j+2} & {\mbox{if }}x\in(y_{*,2}+1,\;z_{*,2}-1),\\
Q_*(z_{*,2})\,(x-z_{*,2}+1) +\zeta_{j+2}\,(z_{*,2}-x)
& {\mbox{if }}x\in[z_{*,2}-1,\;z_{*,2}],\\
Q_*(x)& {\mbox{ if }} x> z_{*,2}.
\end{matrix}
\right. $$
Accordingly, exploiting~\eqref{DIAMOND:EQ},
$$ E(\tilde{Q}) \le
E_{(-\infty,z_{*,1})}(Q_*)+
E_{(y_{*,1},y_{*,2})}(Q_*^{\zeta_{j+1},\zeta_{j+2}}) +
E_{(z_{*,2},+\infty)}(Q_*)+\diamondsuit.$$
Then, since~$(y_{*,1},y_{*,2})\subseteq(z_{*,1},z_{*,2})$,
\begin{equation}\label{8ujkdHHHH1728394:SIKAO}
E(\tilde{Q}) \le
E_{(-\infty,z_{*,1})}(Q_*)+
E_{(z_{*,1},z_{*,2})}(Q_*^{\zeta_{j+1},\zeta_{j+2}}) + 
E_{(z_{*,2},+\infty)}(Q_*)+\diamondsuit.\end{equation}
Also, $\tilde{Q}\in \Gamma(\vec\zeta,\vec{b})$, hence the minimality
of~$Q_*$ gives that
\begin{equation}\label{90iokd68889:0a0}
I(Q_*)\le I(\tilde Q).\end{equation}
Furthermore
\begin{eqnarray*}
\int_{z_{*,1}}^{z_{*,2}} a(x)\,W(\tilde Q(x))\,dx&=&
\int_{y_{*,1}}^{y_{*,2}} a(x)\,W(Q_*^{\zeta_{j+1},\zeta_{j+2}}
(x))\,dx+\diamondsuit\\
&\le&
\int_{z_{*,1}}^{z_{*,2}} a(x)\,W(Q_*^{\zeta_{j+1},\zeta_{j+2}}
(x))\,dx+\diamondsuit.\end{eqnarray*}
This,
\eqref{8ujkdHHHH1728394:SIKAO} and~\eqref{90iokd68889:0a0}
imply that
\begin{eqnarray*}
0 &\le& I(\tilde Q)-I(Q_*)\\
&\le&
E_{(-\infty,z_{*,1})}(Q_*)+
E_{(z_{*,1},z_{*,2})}(Q_*^{\zeta_{j+1},\zeta_{j+2}}) +
E_{(z_{*,2},+\infty)}(Q_*) - E(Q_*)
\\ &&\qquad+\int_{z_{*,1}}^{z_{*,2}} a(x)\,\big[ W(
\tilde Q(x))-
W(Q_*(x))\big]\,dx+\diamondsuit\\
&\le& E_{(z_{*,1},z_{*,2})}(Q_*^{\zeta_{j+1},\zeta_{j+2}})
- E_{(z_{*,1},z_{*,2})}(Q_*)
+\int_{z_{*,1}}^{z_{*,2}} a(x)\,\big[ W(Q_*^{\zeta_{j+1},\zeta_{j+2}}(x))-
W(Q_*(x))\big]\,dx+\diamondsuit.
\end{eqnarray*}
Comparing this with~\eqref{980iodwe6789}, we obtain that~$c\le\diamondsuit$,
which is a contradiction when we make~$\diamondsuit$ as small as we wish
(recall the notation in Remark~\ref{DIAMOND}).
\end{proof}

Now we can construct the desired multibump trajectories:

\begin{theorem}\label{CHS}
Let~$s_0\in\left(\frac12,1\right)$ and~$s\in[s_0,1)$.
Assume that~\eqref{ASS:A:NONDEG} holds.

There exist
$\epsilon_*>0$
and~$b_{2N-2}>b_{2N-3}>\dots>b_2>b_1\in\R$,
possibly depending on~$n$ and
the structural constants
of the kernel and the potential,
such that if~$\epsilon\in(0,\epsilon_*]$,
the following statement holds.

Let~$\zeta_1\in\Z^n$. Let~$\zeta_2\in{\mathcal{A}}(\zeta_1)$,
$\dots$, $\zeta_N\in{\mathcal{A}}(\zeta_{N-1})$.

Then~$Q_*^{\zeta_1,\dots,\zeta_N}$ is a solution of~\eqref{EQUAZ}.
\end{theorem}

\begin{remark}{\rm When~$N=2$, Theorem~\ref{CHS} reduces to Theorem~\ref{HET}.}\end{remark}

\begin{proof}[Proof of Theorem~\ref{CHS}] In view of Lemma~\ref{HJA:AA:2},
we need to show that the trajectory does not hit the constraints.
We argue by contradiction. The idea of the proof is that: first, by 
Lemma~\ref{CLEAN:LEMMA}, we find an integer point close
to the trajectory in a clean interval; then, by
Proposition~\ref{UELO:j}, we localize the integer with respect to
the two integers leading to the excursion of the orbit;
this distinguishes two cases, in one case we use
Proposition~\ref{STICA888} to ``clean'' the orbit to the left
(or to the right), in the other case we will be able to translate
a piece of the orbit and make the energy decrease using~\eqref{ASS:A:NONDEG},
thus obtaining a contradiction.

The details of the argument are the following. We use the short
notation~$Q_*:= Q_*^{\zeta_1,\dots,\zeta_N}$.
By~\eqref{ASR2} and Lemma~\ref{LARGE:EN},
we know that~$I(Q_*)$ is bounded by some~$C_*>0$
(independently on the choice of~$b_1,\dots,b_{2N-2}$).
Thus, we fix~$\rho\in (0,r)$, to be taken sufficiently
small, and we set
$$ L:= \frac{\pi}{12\epsilon}.$$
We suppose that~$\epsilon$ is small enough, such that
\begin{equation} \label{L:DEF}
L\ge \frac{C_*\,|\log\rho|}{\rho^{\frac{4s}{2s-1}}}
,\end{equation}
for a suitably large constant~$C_*$, and we set~$b_1:=L$
and then recursively
\begin{equation}\label{DEF:b:sp}
\begin{split}
&b_{2j}:= b_{2j-1}+22L\\
{\mbox{and }}\quad& b_{2j+1}:=b_{2j}+50L.\end{split}\end{equation}
We suppose, by contradiction,
that there exists~$p_*$ such that one of the following cases holds true:
\begin{eqnarray}
&& \label{CASE1:pp} p_*\in (-\infty,b_1] 
{\mbox{ and }} Q_*(p_*)\in\partial B_r(\zeta_1), \\
&& \label{CASE2:pp} p_*\in [b_{2j},b_{2j+1}] 
{\mbox{ for some~$j\in \{1,\dots, N-2\}$,
and }} Q_*(p_*)\in\partial B_r(\zeta_{j+1}),
\\
&& \label{CASE3:pp} p_*\in [b_{2N-2},+\infty)
{\mbox{ and }} Q_*(p_*)\in\partial B_r(\zeta_N).
\end{eqnarray}
We deal with the cases in~\eqref{CASE1:pp}
and~\eqref{CASE2:pp}, since the case in~\eqref{CASE3:pp}
is similar to the one in~\eqref{CASE1:pp}.

So, let us first suppose that~\eqref{CASE1:pp} holds.
In this case, we observe that~$b_2-b_1=22L$
and so we can use
Lemma~\ref{CLEAN:LEMMA}
(recall~\eqref{L:DEF} and Remark~\ref{56edtycgshd2312989019381})
to find an integer point~$\zeta$ and some
clean point~$x_*\in (b_1+L,b_1+2L)$ for~$Q_*(\cdot-L)$
such that
\begin{equation}\label{Neag78jg:1}
\sup_{x\in\left[ x_*-\frac{|\log\rho|}{2},\;
x_*+\frac{|\log\rho|}{2}\right]}|Q_*(x-L)-\zeta|\le\rho.\end{equation} By
Proposition~\ref{UELO:j}, we know that either~$\zeta=\zeta_1$,
or~$\zeta=\zeta_2$. But indeed~$\zeta\ne\zeta_1$,
otherwise, by~\eqref{ASR5} and Proposition~\ref{STICA888}, we would have that~$|Q_*(x)-\zeta_1|\le r/2$
for any~$x\le x_*$, in contradiction with the assumption taken in~\eqref{CASE1:pp}.

Consequently, we have that
\begin{equation}\label{Neag78jg:2}
\zeta=\zeta_2.
\end{equation}
We also remark that, by 
Lemma~\ref{CLEAN:LEMMA}, there exists
a clean point~$y_*\in [b_2+1,b_2+1+L]$ for~$Q_*$ such that
\begin{equation}\label{Neag78jg:3}
\sup_{x\in\left[ y_*-\frac{|\log\rho|}{2},\;
y_*+\frac{|\log\rho|}{2}\right]}|Q_*(x)-\zeta_2|\le\rho.
\end{equation}
Then, we define
$$ \tilde{Q}(x):=\left\{
\begin{matrix}
Q_*(x-L) & {\mbox{ if }} x\le x_*, \\
Q_*(x_*-L)\,(x_*+1-x)+ \zeta_2\,(x-x_*) & {\mbox{ if }} x\in(x_*,\;x_*+1),\\
\zeta_2 & {\mbox{ if }} x\in[x_*+1,\; y_*-1],\\
\zeta_2\,(y_*-x)+ Q_*(y_*)\,(x-y_*+1)
& {\mbox{ if }} x\in[y_*-1,\;y_*],\\
Q_*(x) & {\mbox{ if }} x> y_*.
\end{matrix}
\right. $$
We point out that
\begin{equation}\label{89iok567654567uzxcj}
\tilde{Q}\in \Gamma(\vec\zeta,\vec{b}).\end{equation}
Indeed, if~$x\le b_1$ then~$x\le x_*$,
and also~$x-L\le b_1$,
hence~$\tilde{Q}(x)=Q_*(x-L)\in \overline{B_r(\zeta_1)}$.
In addition, if~$x\ge b_2$, we have that~$x\ge 23L\ge x_*+1$,
and so~$\tilde{Q}(x)$ always lies in a $\rho$-neighborhood
of~$\zeta_2$, up to~$x=y_*$,
or coincides with~$Q_*$, thus completing the proof of~\eqref{89iok567654567uzxcj}.

{F}rom~\eqref{89iok567654567uzxcj} and the minimality of~$Q_*$, we obtain that
\begin{equation}\label{8ikl34567765fghjkmnbvo99}
\begin{split}
0\;&\le I(\tilde{Q})-I(Q_*)\\
&\le E_{(-\infty,x_*)}(Q_*)+E_{(y_*,+\infty)}(Q_*)-E(Q_*)
\\ &\qquad+\int_{-\infty}^{x_*} a(x)\,W(Q_*(x-L))\,dx
-\int_{-\infty}^{y_*} a(x)\,W(Q_*(x))\,dx
+\diamondsuit \\
&\le \int_{-\infty}^{x_*-L} \big[a(x+L)-a(x)\big]\,W(Q_*(x))\,dx
+\diamondsuit,\end{split}
\end{equation}
where we used the notation in Remark~\ref{DIAMOND}
and~\eqref{DIAMOND:EQ}
(we stress that~\eqref{Neag78jg:1},
\eqref{Neag78jg:2} and~\eqref{Neag78jg:3}
give that the contributions coming
from the linear interpolations are negligible).

Now we use
Lemma~\ref{CLEAN:LEMMA} to find
a clean point~$z_*\in [b_1-L,b_1]$ for~$Q_*$
and so, by~\eqref{ASR5}
and~\eqref{STIMA:AA1},
$$ \overline{a}\,\int_{-\infty}^{b_1-L} W(Q_*(x))\,dx\le \diamondsuit.$$
We insert this into~\eqref{8ikl34567765fghjkmnbvo99}
and we conclude that
$$ 0\le \int_{b_1-L}^{x_*-L} \big[a(x+L)-a(x)\big]\,W(Q_*(x))\,dx
+\diamondsuit.$$
Accordingly, recalling~\eqref{A:EX1},
\begin{equation}\label{987rvd567ijrt67asct8dfs}
0\le -\gamma \int_{b_1-L}^{x_*-L} W(Q_*(x))\,dx
+\diamondsuit,\end{equation}
for some~$\gamma>0$.
Now, $Q_*(b_1-L)$ lies close to~$\zeta_1$, while~$Q_*(x_*-L)$
lies close to~$\zeta_2$ (due to~\eqref{Neag78jg:1}): hence, by continuity and~\eqref{ZERI di W},
we have that~$W(Q_*(x))$ picks up a non-negligible contribution
in a subinterval of~$[b_1-L,x_*-L]$, namely
$$ \int_{b_1-L}^{x_*-L} W(Q_*(x))\,dx\ge c,$$
for some~$c>0$. This and~\eqref{987rvd567ijrt67asct8dfs} imply that~$0\le
-c\gamma+\diamondsuit$, which
is a contradiction when we make~$\diamondsuit$ as small as we wish.
This completes the proof
of Theorem~\ref{CHS} in case~\eqref{CASE1:pp}.\medskip

Now we assume that~\eqref{CASE2:pp} holds true. Then, 
by Lemma~\ref{CLEAN:LEMMA}
(recall~\eqref{L:DEF} and Remark~\ref{56edtycgshd2312989019381}), we know that there exist
clean points~$y_{*,-}\in
\left[b_{2j}+\frac{L}{4},\;b_{2j}+\frac{L}{2}\right]$
and~$y_{*,+}\in\left[b_{2j+1}-\frac{L}{2},b_{2j+1}-\frac{L}{4}\right]$
for~$Q_*$, such that~$|Q_*(y_{*,\pm})-\zeta_{j+1}|\le C\rho$, with~$C>0$.

Hence, by~\eqref{STIMA:AA2},
$$ \sup_{x\in [y_{*,-}, y_{*,+}]} |Q_*(x)-\zeta_{j+1}|\le \frac{r}{2}.$$
This and~\eqref{CASE2:pp} imply that~$p_*\in [b_{2j},y_{*,-}]
\cup [y_{*,+},b_{2j+1}]$.

So, we assume that
\begin{equation}
p_*\in [b_{2j},y_{*,-}],
\end{equation}
the other case being similar. We use again
Lemma~\ref{CLEAN:LEMMA}
to find an integer point~$\zeta$ and some
clean point~$x_*\in \left[
b_{2j}-\frac{L}{2},b_{2j}-\frac{L}{4}\right]$ for~$Q_*$,
such that
\begin{equation}\label{Neag78jg:1:BIS9}
|Q_*(x_*)-\zeta|\le C\rho,\end{equation} with~$C>0$. By
Proposition~\ref{UELO:j}, we know that either~$\zeta=\zeta_j$,
or~$\zeta=\zeta_{j+1}$.

But it cannot be that~$\zeta=\zeta_{j+1}$, otherwise,
by~\eqref{STIMA:AA2},
we would have that
$$ |Q_*(p_*)-\zeta_{j+1}|\le
\sup_{x\in [b_{2j}, b_{2j+1}-L]}|Q_*(x)-\zeta_{j+1}|
\le
\sup_{x\in [x_*,y_{*,+}]}|Q_*(x)-\zeta_{j+1}|\le\frac{r}{2},$$
in contradiction with~\eqref{CASE2:pp}.

Hence, we have that
\begin{equation}\label{Neag78jg:1:BIS9:2}
\zeta=\zeta_j.\end{equation}
Now we use again Lemma~\ref{CLEAN:LEMMA}
to find
a clean point~$z_*\in \left[b_{2j-1}-\frac{L}{2},\;b_{2j-1}-\frac{L}{4}\right]$ for~$Q_*$,
such that
$$ |Q_*(z_*)-\zeta_j|\le C\rho,$$
with~$C>0$. We refer to Figure~\ref{XYZ} for a sketch of
the situation discussed here (of course, the picture
is far from being realistic, since the horizontal scales
involved are much larger than the ones depicted).

\begin{figure}
    \centering
    \includegraphics[width=16.8cm]{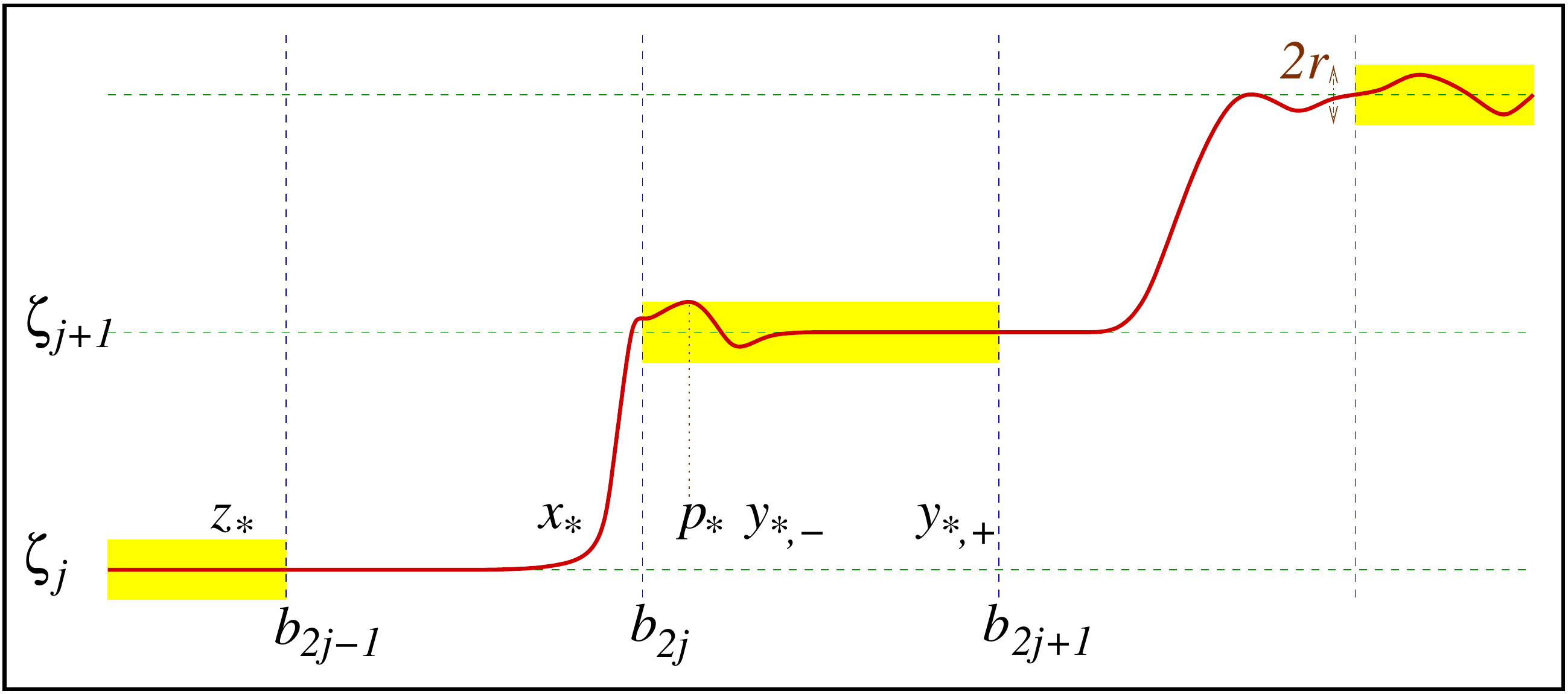}
    \caption{The points~$z_*$, $x_*$, $p_*$, $y_{*,-}$ and~$y_{*,+}$.}
    \label{XYZ}
\end{figure}

In this context, we can
define the following two competitors: we let~$Q_1(x)$ be
$$ \left\{
\begin{matrix}
Q_*(x) & {\mbox{ if }} x\le z_*, \\
Q_*(z_*)\,(z_*+1-x)+ \zeta_j\,(x-z_*) & {\mbox{ if }} x\in(z_*,\;z_*+1),\\
\zeta_j & {\mbox{ if }} x\in[z_*+1,\;x_*-1],\\
\zeta_j\,(x_*-x)+ Q_*(x_*)\,(x-x_*+1)
& {\mbox{ if }} x\in(x_*-1,\;x_*),\\
Q_*(x) & {\mbox{ if }} x\in[x_*,\; y_{*,-}],\\
Q_*( y_{*,-} )\,(y_{*,-}+1-x)+\zeta_{j+1}\,(x-y_{*,-})
& {\mbox{ if }} x\in(y_{*,-},\; y_{*,-}+1),\\
\zeta_{j+1} & {\mbox{ if }} x\in [y_{*,-}+1,\; y_{*,+}-1],\\
Q_*( y_{*,+} )\,(x-y_{*,+}+1)+\zeta_{j+1}\,(y_{*,+}-x)
& {\mbox{ if }} x\in(y_{*,+}-1,\; y_{*,+}),\\
Q_*(x) & {\mbox{ if }} x\ge y_{*,+},
\end{matrix}
\right. $$
and~$Q_2(x)$ be
$$ \left\{
\begin{matrix}
Q_1(x) & {\mbox{ if }} x\le x_*-1-L, \\
Q_1(x_*-1-L)\,(x_*-L-x)+
Q_1(x_*)\,(x-x_*+1+L)& {\mbox{ if }} x\in (x_*-1-L,\;x_*-L),\\
Q_1(x+L) & {\mbox{ if }} x\in [x_*-L, \;y_{*,-}],\\
Q_1(y_{*,-}+L)\,(y_{*,-}+1-x)+  Q_1(y_{*,-}+1)\,(x-y_{*,-})
& {\mbox{ if }} x\in (y_{*,-}, \;y_{*,-}+1),\\
Q_1(x) & {\mbox{ if }} x\ge y_{*,-}+1.
\end{matrix}
\right. $$
We observe that
\begin{equation}\label{8ydhuishHH}
I(Q_1)-I(Q_*) \le\diamondsuit,\end{equation}
thanks to~\eqref{DIAMOND:EQ}.
Also, by inspection, one sees that~$Q_1$, $Q_2\in\Gamma(\vec\zeta,\vec{b})$.
As a consequence, comparing the energy of the minimizer~$Q_*$ with the
one of the competitor~$Q_2$ and using~\eqref{8ydhuishHH},
\begin{equation}\label{89ui789098765678aa}
\begin{split}
0 \;&\le I(Q_2)-I(Q_*)\\
&= I(Q_2)-I(Q_1)+I(Q_1)-I(Q_*)
\\ &\le I(Q_2)-I(Q_1)+\diamondsuit\\
&\le E_{(-\infty, x_*-1-L)}(Q_1)+E_{(x_*-L,y_{*,-})}(Q_1)+
E_{(y_{*,-}+1,+\infty)}(Q_1) -E(Q_1)\\&\qquad
+\int_{x_*-L}^{y_{*,-}}
a(x)\,W(Q_1(x+L))\,dx
-\int_{x_*-1}^{y_{*,-}+1} a(x)\,W(Q_1(x))\,dx
+\diamondsuit 
\\ &\le 
\int_{x_*}^{y_{*,-}+L}
a(x-L)\,W(Q_1(x))\,dx
-\int_{x_*-1-L}^{y_{*,-}+1} a(x)\,W(Q_1(x))\,dx
+\diamondsuit 
.\end{split}
\end{equation}
Now we notice that if~$x\in [y_{*,-}+1,\;y_{*,-}+L]\subseteq
[y_{*,-}+1,\; y_{*,+}-1]$ we have that~$Q_1(x)=\zeta_{j+1}$ and so~$W(Q_1(x))=0$.
Using this information into~\eqref{89ui789098765678aa}, we obtain that
\begin{equation}\label{987HYTFhHHHHJ976yh}
\begin{split}
0\;&\le \int_{x_*}^{y_{*,-}+1}
a(x-L)\,W(Q_1(x))\,dx
-\int_{x_*-1-L}^{y_{*,-}+1} a(x)\,W(Q_1(x))\,dx
+\diamondsuit \\
&\le \int_{x_*}^{y_{*,-}+1}
\big[ a(x-L)-a(x)\big]\,W(Q_1(x))\,dx+\diamondsuit.\end{split}\end{equation}
Now we claim that
\begin{equation}\label{INDlO}
b_{2j}+L\in 24 L \N =\frac{2\pi}{\epsilon}\,\N.\end{equation}
To check this,
we recall~\eqref{DEF:b:sp} and we perform an inductive argument.
Indeed, we have that~$b_2+L= 23L+L=24L$, which checks~\eqref{INDlO}
when~$j=1$. Suppose now that~\eqref{INDlO} holds for some~$j$
and we prove it for the index~$j+1$. For this, we use~\eqref{DEF:b:sp}
to write
$$ b_{2j+2}+L =b_{2j+1}+L+22L=(b_{2j}+L)+50L+22L
\in 24 L \N,$$
as desired.

This proves~\eqref{INDlO}, from which we deduce that
the interval~$[b_{2j}-2L,\;b_{2j}+L]$
is a translation by~$\frac{2\pi k_j}{\epsilon}$ of~$[21L,24L]$,
for some~$k_j\in\N$. This, the periodicity of~$a$ and~\eqref{A:EX2} give that,
for any~$x\in [b_{2j}-2L,\;b_{2j}+L]$,
\begin{equation}\label{098iktg8iujnsdfggGG}
a(x-L)-a(x) \le -\gamma,
\end{equation}
for some~$\gamma>0$. Now, since~$
[x_*,\;y_{*,-}+1]\subseteq [b_{2j}-2L,\;b_{2j}+L]$,
we have that~\eqref{098iktg8iujnsdfggGG} holds for any~$x\in
[x_*,\;y_{*,-}+1]$.

Consequently, by~\eqref{987HYTFhHHHHJ976yh},
\begin{equation}\label{897yihjhgfdsdfgh77} 0\le -\gamma \int_{x_*}^{y_{*,-}+1}
W(Q_1(x))\,dx+\diamondsuit.\end{equation}
Since~$Q_1(x_*)=Q_*(x_*)$, which is close to~$\zeta_j$, by~\eqref{Neag78jg:1:BIS9}
and~\eqref{Neag78jg:1:BIS9:2}, and~$Q_1(y_{*,-}+1)=\zeta_{j+1}$,
it follows that the potential picks up some quantities when going from~$x_*$
to~$y_{*,-}+1$, hence~\eqref{897yihjhgfdsdfgh77} gives that~$0\le -c\gamma+\diamondsuit$,
for some~$c>0$.

This is a contradiction when we take~$\diamondsuit$
appropriately small, hence
we have completed the proof of Theorem~\ref{CHS}.\end{proof}

Now, we obtain Theorem~\ref{TH:MAIN} from
Theorem~\ref{CHS}. 

\begin{appendix}

\section{Proof of Lemma~\ref{SOB}}\label{CAMPA}

We follow the proof given in Section~8 of~\cite{guida}, by
keeping explicit track of the constants involved.

Given~$x_0\in J$ and~$\rho>0$, we define~$J_{x_0,\rho}:=(x_0-\rho,x_0+\rho)\cap J$,
$$ Q_{x_0,\rho}:= 
\frac{1}{|J_{x_0,\rho}|} \int_{J_{x_0,\rho}} Q(y)\,dy$$
and
\begin{equation}\label{9s777yYT}
[Q]_s :=\left(\sup_{{x_0\in J}\atop{\rho>0}}
\rho^{-2s} \int_{J_{x_0,\rho}} |Q(x)-Q_{x_0,\rho}|^2\,dx
\right)^{\frac12}.\end{equation}
First of all, for any~$\xi\in\R^n$ and any~$\rho>0$,
\begin{equation}\label{OSSE00}
|\xi-Q_{x_0,\rho}|^2
=\frac{1}{|J_{x_0,\rho}|^2} 
\left|\int_{J_{x_0,\rho}} \big[ \xi-Q(y)\big]\,dy\right|^2
\le \frac{1}{|J_{x_0,\rho}|}
\int_{J_{x_0,\rho}} \big|\xi-Q(y)\big|^2\,dy.
\end{equation}
Also,
we observe that, for any~$\rho\in(0,1]$,
\begin{equation}\label{OSSE0rho}
|J_{x_0,\rho}|\in [\rho,  \,2\rho].
\end{equation}
Now,
we claim that for any~$\overline{R}\in(0,1]$ and~$\underline{R}
\in(0,\overline{R})$,
\begin{equation}\label{90124J567K}
|Q_{x_0,\overline{R}}-Q_{x_0,\underline{R}}|\le 
\left( \frac{2}{\log 2\,\cdot\,\left( s-\frac12\right)}
+\sqrt{2}\right)\,
[Q]_s \,\overline{R}^{s-\frac12}.
\end{equation}
For this, we fix~$\rho_2>\rho_1>0$, with~$\rho_2\le1$,
we use~\eqref{OSSE00} with~$\xi:=Q_{x_0,\rho_2}$
and~$\rho:=\rho_1$, then we recall~\eqref{OSSE0rho},
and so we obtain that
\begin{equation}\label{G2.23}
\begin{split}
& |Q_{x_0,\rho_2}-Q_{x_0,\rho_1}|^2
\le \frac{1}{|J_{x_0,\rho_1}|}
\int_{J_{x_0,\rho_1}} \big|Q_{x_0,\rho_2}-Q(y)\big|^2\,dy
\\ &\qquad\le
\frac{1}{\rho_1}
\int_{J_{x_0,\rho_2}} \big|Q_{x_0,\rho_2}-Q(y)\big|^2\,dy
\le
\frac{\rho_2^{2s}}{\rho_1}\,[Q]_s^2.
\end{split}\end{equation}
Now we fix~$k\in\N$, $k\ge1$, such that
\begin{equation}\label{0cs78888a}
\frac{1}{2^k}\le \overline{R}^{-1}
\cdot\underline{R} 
\le\frac{1}{2^{k-1}}\end{equation}
and we
define~$R_i:=\overline{R}/2^i$, for any~$i\in\{0,\dots,k\}$.
Notice that
$$ R_k\le \underline{R}\le 2R_k,$$ due to~\eqref{0cs78888a}.
Then, we can use~\eqref{G2.23} with~$\rho_2:=\underline{R}$
and~$\rho_1:=R_k$ and find that
\begin{equation}\label{G2.23:BIS}
|Q_{x_0,\underline{R}}-Q_{x_0,R_k}|
\le
\frac{\underline{R}^{s}}{R_k^{\frac12}}\,[Q]_s\le
\sqrt{2}\,\underline{R}^{s-\frac12}\,[Q]_s.
\end{equation}
Now we use~\eqref{G2.23} with~$\rho_2:=R_{i}$
and~$\rho_1:= R_{i+1}$ and we add up. In this way, we conclude that
\begin{equation}\label{G2.23:BIS:BIS}
\begin{split}
|Q_{x_0,R_0} - Q_{x_0, R_k}|\;
&\le \sum_{i=0}^{k-1} |Q_{x_0,R_{i}} - Q_{x_0, R_{i+1}}|
\le[Q]_s\,\sum_{i=0}^{k-1} \frac{R_{i}^{s}}{R_{i+1}^{\frac12}}
\le \sqrt{2}\,\overline{R}^{s-\frac12}\,
[Q]_s\,\sum_{i=0}^{+\infty} \frac{1}{2^{\left(s-\frac12\right)i}}
\\ &\qquad=\sqrt{2}\,\overline{R}^{s-\frac12}\,
[Q]_s\,\frac{2^{s-\frac12}}{2^{s-\frac12}-1}\le
\frac{2 \,\overline{R}^{s-\frac12}\,
[Q]_s}{\log 2\,\cdot\,\left( s-\frac12\right)}.
\end{split}
\end{equation}
Hence~\eqref{G2.23:BIS}
and~\eqref{G2.23:BIS:BIS} give that
$$|Q_{x_0,R_0} - Q_{x_0,\underline{R}}|
\le
\frac{2 \,\overline{R}^{s-\frac12}\,
[Q]_s}{\log 2\,\cdot\,\left( s-\frac12\right)}+
\sqrt{2}\,\underline{R}^{s-\frac12}\,[Q]_s
\le 
\left( \frac{2}{\log 2\,\cdot\,\left( s-\frac12\right)}
+\sqrt{2}\right)\,
[Q]_s \,\overline{R}^{s-\frac12}.$$
Noticing now that~$R_0=\overline{R}$, we obtain~\eqref{90124J567K},
as desired.

Now we use~\eqref{OSSE00} with~$\xi:=Q(x)$ and we integrate
over~$x\in J_{x_0,\rho}$, to find that
\begin{equation}\label{0oau8JJFR}
\int_{J_{x_0,\rho}} |Q(x)-Q_{x_0,\rho}|^2\,dx
\le \frac{1}{|J_{x_0,\rho}|}
\iint_{J_{x_0,\rho}^2} \big|Q(x)-Q(y)\big|^2\,dx\,dy
\le \frac{1}{\rho}
\iint_{J_{x_0,\rho}^2} 
\big|Q(x)-Q(y)\big|^2\,dx\,dy,\end{equation}
where the last inequality comes from~\eqref{OSSE0rho}.
Notice now that if~$x$, $y\in
J_{x_0,\rho}\subseteq(x_0-\rho,x_0+\rho)$,
then~$|x-y|\le2\rho$. Hence, by~\eqref{0oau8JJFR},
\begin{equation}\label{GG:83}
\begin{split}
\int_{ J_{x_0,\rho} } |Q(x)-Q_{x_0,\rho}|^2\,dx
\;&\le 2^{1+2s}\,\rho^{2s}\,
\iint_{J_{x_0,\rho}^2} 
\frac{\big|Q(x)-Q(y)\big|^2}{|x-y|^{1+2s}}\,dx\,dy\\
&\le 8\,\rho^{2s}\,[Q]_{H^s(J)}^2.\end{split}\end{equation}
By comparing~\eqref{9s777yYT}
with~\eqref{GG:83} we deduce that
\begin{equation}\label{CAMP2}
[Q]_s\le\sqrt{8}\,[Q]_{H^s(J)}.
\end{equation}
{F}rom~\eqref{90124J567K}
and~\eqref{CAMP2}, we obtain that
\begin{equation}\label{098uj6789}
|Q_{x_0,\overline{R}}-Q_{x_0,\underline{R}}|\le
\sqrt{8}
\left( \frac{2}{\log 2\,\cdot\,\left( s-\frac12\right)}
+\sqrt{2}\right)\,[Q]_{H^s(J)}\,\overline{R}^{s-\frac12}
.\end{equation}
Now we claim that
\begin{equation}\label{CONT:Q:J11}
{\mbox{$Q$ is continuous in~$J$.}}
\end{equation}
For this, we use~\eqref{098uj6789}
and the assumption that~$s>\frac12$,
to find that the sequence of functions~$G_\rho(x):=Q_{x,\rho}$
is Cauchy in~$L^\infty(J)$ and so there exists a subsequence~$\rho_j\to0$
such that
\begin{equation}\label{CONT:G:J11:PRE}
{\mbox{$G_{\rho_j}$ converges to some~$G$ uniformly in~$J$,
as~$j\to+\infty$.}}\end{equation}
Now we observe that
\begin{equation}\label{CONT:G:J11}
{\mbox{$G_\rho$ is continuous in~$J$,}}
\end{equation}
for any fixed~$\rho\in(0,1]$. Indeed, we know that~$Q\in L^1(J)$
(see e.g. formula~(6.21) in~\cite{guida}). Therefore,
if~$x_k\in J$ and~$x_k\to x_\infty$
as~$k\to+\infty$, we deduce from the Dominated
Convergence Theorem that
$$ \lim_{k\to+\infty} \frac{1}{|J_{x_\infty,\rho}|}
\int_{J_{x_k,\rho}} Q(y)\,dy = 
\frac{1}{|J_{x_\infty,\rho}|}
\int_{J_{x_\infty,\rho}} Q(y)\,dy.$$
Accordingly
\begin{eqnarray*}
&& \lim_{k\to+\infty} \big|
G_\rho(x_k)-G_\rho(x_\infty)\big|\\
&\le& \lim_{k\to+\infty} 
\left|
\frac{1}{|J_{x_k,\rho}|}
\int_{J_{x_k,\rho}} Q(y)\,dy
-\frac{1}{|J_{x_\infty,\rho}|}
\int_{J_{x_k,\rho}} Q(y)\,dy\right|
+\frac{1}{|J_{x_\infty,\rho}|}
\left|
\int_{J_{x_k,\rho}} Q(y)\,dy
-
\int_{J_{x_\infty,\rho}} Q(y)\,dy\right|\\
&\le& 
\lim_{k\to+\infty} \left| \frac{1}{|J_{x_k,\rho}|}
-\frac{1}{|J_{x_\infty,\rho}|}\right|\,
\int_{J} Q(y)\,dy\\
&=&0,\end{eqnarray*}
and this gives~\eqref{CONT:G:J11}.

By~\eqref{CONT:G:J11:PRE} and~\eqref{CONT:G:J11},
we obtain that
\begin{equation}\label{8s8III}
{\mbox{$G$ is continuous.}}\end{equation}
Now, for any~$x$ in the interior of the segment~$J$,
we have that~$J_{x,\rho_j} = (x-\rho_j,x+\rho_j)$
if $j$ is large enough and so, if~$x$ is also a Lebesgue point
for~$Q$,
\begin{eqnarray*}
&& G(x) = \lim_{\rho_j\to0} G_{\rho_j}(x)
=\lim_{\rho_j\to0} Q_{x,\rho_j}
=\lim_{\rho_j\to0} 
\frac{1}{|J_{x,\rho_j}|} \int_{J_{x,\rho_j}} Q(y)\,dy\\
&&\qquad=\lim_{\rho_j\to0}
\frac{1}{2\rho_j} \int_{x-\rho_j}^{x+\rho_j} Q(y)\,dy
= Q(x).\end{eqnarray*}
Accordingly, $Q$ and~$G$ coincide in all the Lebesgue
points of the interior of~$J$ and thus almost everywhere in~$J$.
Hence, from~\eqref{8s8III} (and possibly redefining~$Q$
in a negligible set), we conclude that~\eqref{CONT:Q:J11}
holds true.

Thanks to~\eqref{CONT:Q:J11}, we can now send~$\underline{R}\to0$
in~\eqref{098uj6789} and obtain that
\begin{equation}\label{90k89000}
|Q_{x_0,\overline{R}}-Q(x_0)|\le
\sqrt{8}
\left( \frac{2}{\log 2\,\cdot\,\left( s-\frac12\right)}
+\sqrt{2}\right)\,[Q]_{H^s(J)}\,
\overline{R}^{s-\frac12},\end{equation}
for any~$\overline{R}\in(0,1]$ and~$x_0\in J$.

Now we fix~$X$, $Y\in J$ and we take~$\overline{R}:=2|X-Y|$.
Then, we obtain from~\eqref{90k89000} (applied
with~$x_0:=X$ and with~$x_0:=Y$) that
\begin{equation}\label{90k89000:BIS}
|Q(X)-Q_{X,\overline{R}}|+
|Q_{Y,\overline{R}}-Q(Y)|
\le
8\,
\left( \frac{2}{\log 2\,\cdot\,\left( s-\frac12\right)}
+\sqrt{2}\right)\,[Q]_{H^s(J)}\,
|X-Y|^{s-\frac12}.\end{equation}
Now we take~$P:=\frac{X+Y}{2}$ 
and we notice that~$(P-\overline{R},P+\overline{R})$
contains the segment joining~$X$ and~$Y$,
which lies in~$J$ and has length~$\overline{R}/2$, therefore
\begin{equation}\label{9044GHjKK}
|J_{P,\overline{R}}|\ge \frac{ {\overline{R}} }{2}.
\end{equation}
Now we fix~$z\in J_{P,\overline{R}}$.
By~\eqref{OSSE00}, used here with~$x_0:=X$ and~$
\rho:=\overline{R}$ and~$\xi:=Q(z)$, we see that
$$ |Q(z)-Q_{X,\overline{R}}|^2
\le \frac{1}{|J_{X,\overline{R}}|}
\int_{J_{X,\overline{R}}} \big|Q(z)-Q(y)\big|^2\,dy.$$
Now we observe that~$\overline{R}\le2$ and so, by~\eqref{OSSE0rho},
$$ |J_{X,\overline{R}}| \ge |J_{X,\overline{R}/2}|\ge\frac{\overline{R}}{2}$$
and therefore
$$ |Q(z)-Q_{X,\overline{R}}|^2
\le \frac{\;2\;}{\overline{R}}
\int_{ J_{X,\overline{R}} } \big|Q(z)-Q(y)\big|^2\,dy
\le
\frac{\;2\;}{\overline{R}}
\int_{ J_{P,2\overline{R}} } \big|Q(z)-Q(y)\big|^2\,dy
.$$
Similarly
$$ |Q(z)-Q_{Y,\overline{R}}|^2
\le \frac{\;2\;}{\overline{R}}
\int_{ J_{P,2\overline{R}} } \big|Q(z)-Q(y)\big|^2\,dy.$$
Therefore
\begin{eqnarray*}
|Q_{X,\overline{R}}-Q_{Y,\overline{R}}|^2 &\le&
2\Big( |Q_{X,\overline{R}}-Q(z)|^2+
|Q(z)-Q_{Y,\overline{R}}|^2
\Big)\\
&\le& \frac{\;8\;}{\overline{R}}
\int_{J_{P,2\overline{R}}} \big|Q(z)-Q(y)\big|^2\,dy
.\end{eqnarray*}
Thus, by integrating over~$z\in J(P,\overline{R})$
and recalling~\eqref{9044GHjKK},
$$ \frac{\overline{R}}{\;2\;}\,
|Q_{X,\overline{R}}-Q_{Y,\overline{R}}|^2
\le \frac{\;8\;}{\overline{R}}
\iint_{J_{P,2\overline{R}}^2} \big|Q(z)-Q(y)\big|^2\,dz\,dy.$$
As a consequence
\begin{eqnarray*}
&& |Q_{X,\overline{R}}-Q_{Y,\overline{R}}|^2
\le \frac{\;16\;}{\overline{R}^2}
\iint_{ {J_{P,2\overline{R}}^2} } \big|Q(z)-Q(y)\big|^2\,dz\,dy\\
&&\qquad\le
\frac{\;16\;}{\overline{R}^2}
\iint_{ {J_{P,2\overline{R}}^2} } (4\overline{R})^{1+2s}
\frac{\big|Q(z)-Q(y)\big|^2}{|z-y|^{1+2s}}\,dz\,dy
\le 4^{3+2s} \,\overline{R}^{2s-1} \,[Q]^2_{H^s(J)}\le
4^{6} \,|X-Y|^{2s-1} \,[Q]^2_{H^s(J)}.\end{eqnarray*}
Using this and~\eqref{90k89000:BIS}, we obtain that
\begin{eqnarray*}
|Q(X)-Q(Y)|&\le&|Q(X)-Q_{X,\overline{R}}|+|Q_{X,\overline{R}}-
Q_{Y,\overline{R}}|+
|Q_{Y,\overline{R}}-Q(Y)|
\\ &\le&
8\,
\left( \frac{2}{\log 2\,\cdot\,\left( s-\frac12\right)}
+4\right)\,[Q]_{H^s(J)}\,
|X-Y|^{s-\frac12}.\end{eqnarray*}
This proves~\eqref{LA:99:001}.

\section{Proof of Lemma~\ref{LAKK}}\label{MATTEO}

We notice that~$Q\in C^{0,s-\frac12}([0,1))$,
thanks to Lemma~\ref{SOB}, hence the condition~$Q(0)=0$
is attained continuously and, more precisely, for any~$y\in[0,1]$,
$$ |Q(y)|\le S_0\,[Q]_{H^s([0,1))}\, |y|^{s-\frac12}.$$
Accordingly, if we define
$$ V(x):=\frac{1}{x}\int_0^x \big( Q(x)-Q(y)\big)\,dy
= Q(x)-\frac{1}{x}\int_0^x Q(y)\,dy,$$
we have that, for any~$x\in[0,1]$,
\begin{equation}\label{7udw777PP}
|V(x)|\le S_0\, [Q]_{H^s([0,1))}\,\Big( |x|^{s-\frac12}+
\frac{1}{x}\int_0^x |y|^{s-\frac12}\,dy\Big)=C\,S_0\,
|x|^{s-\frac12} ,\end{equation}
for some~$C>0$.
Moreover, by H\"older inequality,
$$ |V(x)|^2\le \frac{1}{x}\int_0^x \big| Q(x)-Q(y)\big|^2 \,dy.$$
We also notice that if~$y\in [0,x]$ then~$x\ge x-y=|x-y|$.
As a consequence, 
\begin{equation}\label{V:X:ST:0}
\begin{split}
& \int_0^{\beta} 
x^{-2s} |V(x)|^2\,dx
\le \int_0^{\beta} x^{-1-2s} 
\left[ \int_0^x \big| Q(x)-Q(y)\big|^2 \,dy\right]\,dx
\\ &\qquad\le \int_0^{\beta} 
\left[ \int_0^x |x-y|^{-1-2s}\big| Q(x)-Q(y)\big|^2 \,dy\right]\,dx
.\end{split}
\end{equation}
Furthermore,
$$ \int_{\beta}^{+\infty} x^{-2s} |V(x)|^2\,dx
\le \frac{ \|V\|_{L^\infty((0,+\infty),\R^n)} }{(2s-1)\,\beta^{2s-1} }.$$
Hence, noticing that~$\|V\|_{L^\infty((0,+\infty),\R^n)}\le 2\|Q\|_{L^\infty((0,+\infty),\R^n)}$,
we find that
$$ \int_{\beta}^{+\infty} x^{-2s} |V(x)|^2\,dx
\le \frac{ 2\|Q\|_{L^\infty((0,+\infty),\R^n)} }{(2s-1)\,\beta^{2s-1} }.$$
{F}rom this and~\eqref{V:X:ST:0}, we obtain that
\begin{equation}\label{V:X:ST}
\int_0^{+\infty}
x^{-2s} |V(x)|^2\,dx
\le 
\iint_{(0,\beta)\times(0,x)}
\frac{\big| Q(x)-Q(y)\big|^2}{|x-y|^{1+2s}} \,dx\,dy+
\frac{ 2\|Q\|_{L^\infty((0,+\infty),\R^n)} }{(2s-1)\,\beta^{2s-1} }.
\end{equation}
Now we recall a classical inequality due to Hardy, namely
that for any~$\alpha>0$ and any measurable function~$f$,
we have that
\begin{equation}\label{HAR}
\int_0^{+\infty} 
x^{-1-2\alpha } \left[ \int_0^x y^{-1}\,|f(y)|\,dy \right]^{2}\,dx
\le \alpha^{-2}\,\int_0^{+\infty}
y^{-1-2\alpha }\,|f(y)|^2\,dy .
\end{equation}
To prove it, we make the substitution~$y=tx$ twice
and we apply the Minkowski integral inequality to
the function~$g(x,t):= x^{-\frac12 -\alpha} t^{-1}\,|f(tx)|$.
In this way, we obtain that
\begin{eqnarray*}
&& \int_0^{+\infty}
x^{-1-2\alpha } \left[ \int_0^x y^{-1}\,|f(y)|\,dy \right]^{2}\,dx
=
\int_0^{+\infty}
x^{-1-2\alpha } \left[ \int_0^1 t^{-1}\,|f(tx)|\,dt \right]^{2}\,dx\\
&&\qquad=
\int_0^{+\infty}
\left[ \int_0^1 g(x,t)\,dt \right]^{2}\,dx
\le \left[ \int_0^{1}
\left[ \int_0^{+\infty} |g(x,t)|^2\,dx \right]^{\frac12}\,dt \right]^2
\\ &&\qquad=
\left[ \int_0^{1}
\left[ \int_0^{+\infty} 
x^{-1-2\alpha } t^{-2}\,|f(tx)|^2\,dx \right]^{\frac12}\,dt \right]^2
\\ &&\qquad=
\left[ \int_0^{1}
\left[ \int_0^{+\infty}
y^{-1-2\alpha } t^{2\alpha -2}\,|f(y)|^2\,dy \right]^{\frac12}\,dt \right]^2
\\ &&\qquad=\left[ \int_0^{1} t^{\alpha-1}
\left[ \int_0^{+\infty}
y^{-1-2\alpha } \,|f(y)|^2\,dy \right]^{\frac12}\,dt \right]^2
\\ &&\qquad= \frac{1}{\alpha^2}
\int_0^{+\infty}
y^{-1-2\alpha } \,|f(y)|^2\,dy . 
\end{eqnarray*}
This proves~\eqref{HAR}.

Now we use~\eqref{HAR} with~$f:=V$ and~$\alpha:=s-\frac12$ and we obtain
that
\begin{equation}\label{HAR:2}
\int_0^{+\infty}
x^{-2s} \left[ \int_0^x y^{-1}\,|V(y)|\,dy \right]^{2}\,dx
\le \frac{4}{(2s-1)^2}\,\int_0^{+\infty}
y^{-2s}\,|V(y)|^2\,dy .
\end{equation}
Now we define
$$ Z(x):= \int_0^x y^{-1}\,V(y)\,dy$$
and we deduce from~\eqref{HAR:2} that
\begin{equation}\label{HAR:3}
\int_0^{+\infty}
x^{-2s }\, |Z(x)|^{2}\,dx
\le \frac{4}{(2s-1)^2}\,\int_0^{+\infty}
y^{-2s }\,|V(y)|^2\,dy .
\end{equation}
Also, recalling~\eqref{7udw777PP}, we have that, for any~$x\in[0,1]$,
$|Z(x)|$ is controlled by~$|x|^{s-\frac12}$, which gives that~$Z(0)=0$. Hence, if
we define
$$ F(x):= V(x)+Z(x)-Q(x),$$
recalling again~\eqref{7udw777PP} we find that~$F(0)=0$.
Moreover,
$$ F'(x)= 
Q'(x)+\frac{1}{x^2}\int_0^x Q(y)\,dy
-\frac{Q(x)}{x} +
\frac{V(x) }{x} -Q'(x)=0.$$
As a consequence, $F$ is constantly equal to zero
in~$[0,+\infty)$, which says that
$$ Q(x)= V(x)+Z(x),$$
for any~$x\ge0$. This implies that
$$ |Q(x)|^2 \le \big( |V(x)| +|Z(x)|\big)^2
\le 2\big( |V(x)|^2 +|Z(x)|^2\big).$$
Therefore, by~\eqref{HAR:3},
\begin{eqnarray*}
\int_0^{+\infty}
x^{-2s }\, |Q(x)|^{2}\,dx &\le&
2\left(
\int_0^{+\infty}
x^{-2s }\, |V(x)|^{2}\,dx+
\int_0^{+\infty}
x^{-2s }\, |Z(x)|^{2}\,dx
\right) \\
&\le& 2\left(1+
\frac{4}{(2s-1)^2}\right)\,\int_0^{+\infty}
y^{-2s }\,|V(y)|^2\,dy .\end{eqnarray*}
This and~\eqref{V:X:ST} imply the thesis of
Lemma~\ref{LAKK}.

\end{appendix}

\section*{Acknowledgments}
It is a pleasure to thank Matteo Cozzi for very useful conversations
and the University of Texas at Austin for the warm hospitality.
 
This work has been supported by Alexander von Humboldt Foundation, 
NSF grant DMS-1262411 ``Regularity and stability results in variational problems'' and
ERC grant 277749 ``EPSILON Elliptic PDE's and Symmetry of Interfaces and Layers for Odd Nonlinearities". 

\bibliography{rab}
\bibliographystyle{is-alpha}

\vspace{2mm}

\end{document}